\documentclass[final,onefignum,onetabnum]{siamart190516}

\usepackage{braket,amsfonts}

\usepackage{array}
\usepackage[utf8]{inputenc} 
\usepackage[T1]{fontenc}    

\usepackage{palatino} 
\usepackage{amsmath}
\usepackage{paralist}
\usepackage{graphicx}
\usepackage{epstopdf}
\usepackage{amssymb}
\usepackage{amsfonts}       
\usepackage{nicefrac}       
\usepackage{epstopdf}
\usepackage{verbatim}
\usepackage{mathrsfs}
\usepackage{mathtools}
\usepackage{pstricks}
\usepackage{xcolor}
\usepackage{enumerate}
\usepackage{bm}

\usepackage[title]{appendix}
\usepackage{url}            
\usepackage[normalem]{ulem}
\usepackage{booktabs}       

\newcommand{\bmat}[1]{\begin{bmatrix}#1\end{bmatrix}}
\newcommand{\clred}[1]{{\color{black}#1}}
\newsiamthm{claim}{Claim}
\newsiamremark{remark}{Remark}
\newsiamremark{hypothesis}{Hypothesis}
\crefname{hypothesis}{Hypothesis}{Hypotheses}
\newtheorem{ex}{Example}

\usepackage{xspace}
\usepackage{bold-extra}
\usepackage[most]{tcolorbox}

\colorlet{texcscolor}{blue!50!black}
\colorlet{texemcolor}{red!70!black}
\colorlet{texpreamble}{red!70!black}
\colorlet{codebackground}{black!25!white!25}

\lstdefinestyle{siamlatex}{%
  style=tcblatex,
  texcsstyle=*\color{texcscolor},
  texcsstyle=[2]\color{texemcolor},
  keywordstyle=[2]\color{texemcolor},
  moretexcs={cref,Cref,maketitle,mathcal,text,headers,email,url},
}

\tcbset{%
  colframe=black!75!white!75,
  coltitle=white,
  colback=codebackground, 
  colbacklower=white, 
  fonttitle=\bfseries,
  arc=0pt,outer arc=0pt,
  top=1pt,bottom=1pt,left=1mm,right=1mm,middle=1mm,boxsep=1mm,
  leftrule=0.3mm,rightrule=0.3mm,toprule=0.3mm,bottomrule=0.3mm,
  listing options={style=siamlatex}
}

\newtcblisting[use counter=example]{example}[2][]{%
  title={Example~\thetcbcounter: #2},#1}

\newtcbinputlisting[use counter=example]{\examplefile}[3][]{%
  title={Example~\thetcbcounter: #2},listing file={#3},#1}

\DeclareTotalTCBox{\code}{ v O{} }
{ 
  fontupper=\ttfamily\color{black},
  nobeforeafter,
  tcbox raise base,
  colback=codebackground,colframe=white,
  top=0pt,bottom=0pt,left=0mm,right=0mm,
  leftrule=0pt,rightrule=0pt,toprule=0mm,bottomrule=0mm,
  boxsep=0.5mm,
  #2}{#1}

\patchcmd\newpage{\vfil}{}{}{}
\begin{tcbverbatimwrite}{tmp_\jobname_header.tex}
\title{When does stabilizability imply the existence of infinite horizon optimal control in nonlinear systems?\thanks{%
Submitted to the editors DATE.%
\funding{The work was supported, in part, by JSPS KAKENHI Grant Numbers
JP26289128 and Nanzan University Pache Research Subsidy I-A-2 for the 2020 academic year.}}}

\author{
Noboru Sakamoto\thanks{
Faculty of Science and Engineering, Nanzan University, Yamazato-cho, Showa-ku, Nagoya, 464-8673, Japan 
(\email{noboru.sakamoto@nanzan-u.ac.jp}).}
}

\headers{When does stabilizability imply the existence of optimal control?}{Noboru Sakamoto}
\end{tcbverbatimwrite}
\input{tmp_\jobname_header.tex}

\ifpdf
\hypersetup{ pdftitle={Guide to Using  SIAM'S \LaTeX\ Style} }
\fi

\begin{document}
\maketitle

\begin{tcbverbatimwrite}{tmp_\jobname_abstract.tex}
\begin{abstract}
The paper addresses an existence problem for infinite horizon optimal control when the system under control is exponentially stabilizable or stable. Classes of nonlinear control systems for which infinite horizon optimal controls exist are identified in terms of stability, stabilizability, detectability and growth conditions. The result then applies to estimate the existence region of stable manifolds in the associated Hamiltonian systems. Applications of the results also include the analysis for turnpike property in nonlinear finite horizon optimal control problems by a geometric approach.  
\end{abstract}


\end{tcbverbatimwrite}
\input{tmp_\jobname_abstract.tex}
\section{Introduction}
Optimal control problems (OCPs) are of significance from mathematical and engineering viewpoints, as applications and extensions of Calculus of Variations as well as design tools for systems describing engineering processes. There are two approaches to OCPs, one from \clred{the} sufficiency of optimality (Dynamic Programming \cite{Bellman:57:DP} developed by Bellman) and the other from necessity (Maximum Principle \cite{Pontryagin:62:MTOP} developed by Pontryagin). We refer to \cite{Athans:66:OC,Bryson:75:AOC,Cesari:83:OTA,Liberzon:12:CVOCT} for the theory of OCPs and to \cite{Bryson:96:cssmag} for a survey on OCPs from mathematical and engineering viewpoints. 
\clred{OCPs for infinite horizon are of special interest in engineering, such as linear quadratic regulator problems, since the stability issues are inherently involved in such problems.}

In the Dynamic Programming approach for infinite horizon OCPs, one derives a nonlinear partial differential equation, called {\em Hamilton-Jacobi-Bellman equation} (HJBE), the solution of which gives an optimal control as a feedback law. There is \clred{a} large amount of research on \clred{the} solution method for HJBEs, for which we refer to \cite{Al'brekht:61:jamm,Lukes:69:sicon,Navasca2007251,Aguilar2014527} for Taylor expansion method and to \cite{Kreisselmeier:94:ieee-ac,Beard:97:automatica,Beard:98:ijc,Sakamoto:08:ieee-ac,Ohtsuka:11:ieee-ac,Sakamoto:13:automatica} for other numerical or algebraic approaches (see \cite{Beeler:00:jota} for a survey on the numerical methods for HJBE). Interestingly, when one applies these methods, no information for the solvability region is available and only local solvability around an equilibrium is examined, which amounts to the stabilizability and \clred{the detectability} of the linear part. 
For instance, swing-up and stabilization \clred{feedbacks} for inverted pendulum and acrobot are obtained in \cite{Horibe:17:ieee_cst,Horibe:19:ieee_cst} by numerically solving HJBEs. However, no theory a priori guarantees the solvability of the OCPs for initial pending positions. 

In this paper, motivated \clred{by} linear control theory, we wish to clarify under what conditions stabilizability (or stability) guarantees the solvability of infinite horizon OCPs. In our study, we restrict ourselves to \clred{an} affine nonlinear control system and to a cost functional that consists of a quadratic term on inputs and a nonnegative penalty function on states. 
We make full use of these structures to prove that if \clred{the} free dynamics is globally exponentially stable and \clred{the} input matrix is bounded, an optimal control exists globally and that if the control system is exponentially stabilizable, certain growth conditions at infinity are satisfied and detectability and coercivity conditions on the penalty function on the states are satisfied, then, an optimal control exists in the stabilizable region. 
\clred{
The results allow us to utilize a large number of works in nonlinear stability and stabilizability such as feedback linearization, the notion of zero dynamics and backstepping method in \cite{Nijmeijer:90:NDCS,Isidori:95:NCS,Khalil:92:NS,vanderSchaft:17:LGPTNC,Sepulchre:97:CNC} in the analysis for OCPs (see \S~\ref{sctn:applications}).
}

Another motivation of the paper arises from turnpike phenomena in optimal control. It is often observed that under certain conditions optimal control and corresponding trajectory for finite (but long) horizon \clred{problems} are exponentially close to their steady-state optimum counterparts most of \clred{the} time in the control process except for the beginning and the end \clred{in} thin intervals. Turnpike is a metaphor used in econometrics \cite{Mckenzie:63:econometrica} for this behavior of optimally controlled systems as, when traveling from one place to a distant place, we always take \clred{the} highway to cover the distance at the best rate\cite{Dorfman:58:LPEA}. In control theory, this property is observed first in \cite{Wilde:72:ieeetac,Rockafellar:73:jota} as dichotomy or saddle point property. We refer to \cite{Carlson:91:IHOC,Zaslavski:06:TPCVOC} for general accounts on turnpike theory in control systems. The turnpike phenomena are investigated from nonlinear control \cite{Anderson:87:automatica,Trelat:15:jde}, Hamilton-Jacobi theoretic \cite{Rapaport:04:esaim}, PDEs \cite{Porretta:13:sicon,Porretta:16:INdAM,Gugat:16:syscon,Trelat:18:sicon,Zuazua:17:arc}, dissipative system theoretic \cite{Damm:14:sicon,Grune:16:syscon,Faulwasser:17:automatica,Berberich:18:ieeecst,Grune:18:sicon} and geometric (or dynamical system theoretic) \cite{sakamoto:20:prep_cdc} viewpoints. The framework in \cite{sakamoto:20:prep_cdc} to study the turnpike is based on stable and unstable manifolds of a Hamiltonian system associated with an OCP. It is shown that turnpike if the stable and unstable manifolds satisfy certain conditions.
Under some conditions on the linear part of systems, \clred{\S~\ref{sctn:stabl_manifold} of the present paper shows that} the problem to find solvability region for OCPs is equivalent to estimate the existence region of a stable manifold in the base space (control space) for associated Hamiltonian systems. \clred{This analysis for the stable manifold is closely related to the transversality condition for infinite horizon OCPs, for which the papers such as \cite{Aseev20041094,Aseev20071,Cannarsa20181188} consider under more general conditions than those in the present paper.}

The structure of the paper is as follows. In \S~\ref{sctn:existence_OC}, the Direct Method of Calculus of Variations is applied to show the existence of optimal control for exponentially stable and stabilizable cases. In \S~\ref{sctn:stabl_manifold}, the existence of costates $p(t)$ defined on $[0,\infty)$ with $p(\infty)=0$ is shown, which is equivalent to the existence of stable manifold. 
In \S~\ref{sctn:applications}, we show several classes of nonlinear systems in which the results in this paper are applicable. One of them is a class where stable and unstable manifolds in associated Hamiltonian systems exist with a canonical projection property to the base space, and therefore, the turnpike occurs in finite interval OCPs.

\section{Existence of optimal control}\label{sctn:existence_OC}
Let us consider a nonlinear control system of the form
\begin{equation}
    \dot{x}=f(x)+g(x)u,\ x(0)=x_0,\label{eqn:n_sys}
\end{equation}
where $f:\mathbb{R}^n\to\mathbb{R}^n$ and $g:\mathbb{R}^n\to \mathbb{R}^{n\times m}$ are $C^2$ maps. Let us assume that $x=0$ is an equilibrium of $\dot x=f(x)$; $f(0)=0$. The OCP for (\ref{eqn:n_sys}) is to find a control input $u$ that minimizes a cost functional. In this paper, we consider a cost functional of the form
\begin{equation}
    J=\int_0^\infty |u(t)|^2/2+h(x(t))\,dt \label{eqn:cost}
\end{equation}
with control set $L^2((0,\infty);\mathbb{R}^m)$. The function $h:\mathbb{R}^n\to\mathbb{R}$ is a locally Lipschitz nonnegative function $h(x)\geqslant0$ with $h(0)=0$ that penalizes $x$ in the process of control. $J$ is a functional on $L^2((0,\infty);\mathbb{R}^m)$ taking values in $\mathbb{R}^+\cup \{+\infty\}$ and we denote its value for $u$ by $J(u)$. When a solution to (\ref{eqn:n_sys}) for a $u\in L^2((0,\infty);\mathbb{R}^m)$ has finite escape time, we set $J(u)=\infty$. 


The first problem we tackle in the paper is to determine the conditions under which stability/stabilizability of (\ref{eqn:n_sys}) guarantees the existence of optimal control (\ref{eqn:n_sys})-(\ref{eqn:cost}). 
Throughout this section, the initial condition $x(0)=x_0$ for (\ref{eqn:n_sys}) is fixed. Since input function is not assumed to be piecewise continuous, we use a generalized notion of solution for ordinary differential equations (ODEs). For an ODE, a Carath\'eodory solution is an absolute continuous function defined on an interval $I\subset\mathbb{R}$ such that it satisfies the ODE except on a subset of $I$ which has zero Lebesgue measure (see, e.g., page 28 of \cite{Hale:73:ODE}). 
\begin{lemma}
For each $u\in L^2((0,\infty);\mathbb{R}^m)$ and $t_0\geqslant0$, there exists a unique solution passing through $(t_0,x_0)$ in the sense of \clred{Carath\'eodory} for (\ref{eqn:n_sys}). 
\end{lemma}
\begin{proof}
Take an arbitrary compact set $K\subset \mathbb{R}^n$. Then, \clred{for $x\in K$,}
\begin{align}
&|f(x)+g(x)u(t)|\leqslant \sup_{K}|f(x)|+\sup_{K}\|g(x)\|\,|u(t)|\label{eqn:carathe1}\\
&
\begin{aligned}
|f(x_1)+g(x_1)u(t) - f(x_2)-g(x_2)u(t)|&\leqslant |f(x_1)-f(x_2)|+\|g(x_1)-g(x_2)\|\,|u(t)|\\
&\leqslant (M_f + M_g|u(t)|)|x_1-x_2|,
\end{aligned}\label{eqn:carathe2}
\end{align}
where $\|\cdot\|$ is the matrix induced norm and $M_f$, $M_g$ are constants such that 
\[
|f(x_1)-f(x_2)|\leqslant M_f|x_1-x_2|, \quad 
\|g(x_1)-g(x_2)\|\leqslant M_g|x_1-x_2|
\]
for all $x_1$, $x_2$ in $K$. The right hands of (\ref{eqn:carathe1}), (\ref{eqn:carathe2}) are locally integrable functions of $t$. Therefore, from Theorems~5.1 and 5.3 in \S~1 of \cite{Hale:73:ODE}, there exists a unique solution that is absolute continuous. 
\end{proof}
For $u\in L^2((0,\infty);\mathbb{R}^m)$, let us denote the corresponding solution for (\ref{eqn:n_sys}) by \clred{$x_u$}. 
\begin{proposition}\label{prop:minimizer} 
Assume that for each bounded set $\mathscr{U}$ in $L^2((0,\infty);\mathbb{R}^m)$, the corresponding set 
$\{x_u\,|\, u\in\mathscr{U} \}$ is bounded in $C^0([0,T];\mathbb{R}^n)$ for any $T>0$. 
Assume also that there exists a $u\in L^2((0,\infty);\mathbb{R}^m)$ such that $J(u)<+\infty$. Then, there exists a $\bar u \in L^2((0,\infty);\mathbb{R}^m)$ such that $J(\bar u)=\inf_{L^2((0,\infty);\mathbb{R}^m)}J$. 
\end{proposition}
\begin{proof}
\underline{Step 1.} From the existence of a $u\in L^2((0,\infty);\mathbb{R}^m)$ such that $J(u)<\infty$, there exists a minimizing sequence $\{u_m\}\subset L^2((0,\infty);\mathbb{R}^m)$. For sufficiently large $m$, we have $J(u_m)<1+\inf_{L^2((0,\infty);\mathbb{R}^m)}J$ and therefore, $\{u_m\}$ is a bounded set. By Banach-Alaoglu Theorem, up to subsequence, $u_m$ weakly converges to a $\bar u \in L^2((0,\infty);\mathbb{R}^m)$. Let $x_m:=\clred{x_{u_m}}$. Then, from the assumption, $\{x_m\}$ is uniformly bounded. \clred{Take an arbitrary $T>0$ and} we will show that $x_m$ is equicontinuous \clred{on $[0,T]$}. Let us take arbitrary $t_1<t_2$ \clred{in $[0,T]$}. Then, 
\begin{align*}
|x_m(t_2)-x_m(t_1)|&\leqslant\int_{t_1}^{t_2}|f(x_m(s))|+\|g(x_m(s))\|\,|u_m(s)|\,ds\\
&\leqslant C_1|t_2-t_1|+C_2\int_{t_1}^{t_2}|u_m(s)|\,ds\\
&\leqslant C_1|t_2-t_1|+C_2\sqrt{|t_2-t_1|}\|u_m\|_{L^2((0,\infty);\mathbb{R}^m)}\\
&\leqslant C_1|t_2-t_1|+C_2 \sup_{m\in\mathbb{N}}\|u_m\|_{L^2((0,\infty);\mathbb{R}^m)}\sqrt{|t_2-t_1|},
\end{align*}
where we have taken constants $C_1$, $C_2>0$, \clred{which may depend on $T$,} such that 
\[
|f(x_m(t))|<C_1, \quad \|g(x_m(t))\|<C_2 \quad \text{for }t\in [0,T], m\in\mathbb{N}
\]
by using the uniform boundedness of $x_m$. This proves that $x_m$ is equicontinuous. From Ascoli-Arzel\'a Theorem, up to subsequence $x_m$ uniformly converges to an $\bar x\in C^0([0,\infty);\mathbb{R}^n)$ on $[0,T]$. 

Next, we prove that $\bar x = \clred{x_u}$. Since 
\[
x_m(t)=x_0+\int_0^tf(x_m(s))+g(x_m(s))u_m(s)\,ds,
\]
and uniform convergence of $x_m$ to $\bar x$, it suffices to prove that 
\[
\int_0^tg(x_m(s))u_m(s)\,ds\to \int_0^tg(\bar x(s))\bar u(s)\,ds
\quad \text{as }m\to\infty
\]
from the uniqueness of solutions of initial value problems. 
First we note that
\begin{align}
\left| \int_0^tg(x_m(s))u_m(s)-g(\bar x(s))u_m(s)\,ds \right|
&\leqslant \int_0^t\|g(x_m(s))-g(\bar x(s))\|\,|u_m(s)|\,ds\nonumber\\
&\leqslant \left(\int_0^t\|g(x_m(s))-g(\bar x(s))\|^2\,ds\right)^{1/2} \left(\int_0^t|u_m(s)|^2\,ds\right)^{1/2}\nonumber\\
&\leqslant\sup_{m\in\mathbb{N}}\|u_m\|_{L^2((0,\infty);\mathbb{R}^m)}  \left(\int_0^t\|g(x_m(s))-g(\bar x(s))\|^2\,ds\right)^{1/2}\nonumber\\
&\to 0 \quad \text{as}\ m\to\infty. \label{ineq:conv_gmum}
\end{align}
Let $[\,\cdot\,]_j$ denote the $j$-th component of a vector for $j=1,\ldots,n$. Then, the operator 
\[
u\in L^2((0,\infty);\mathbb{R}^m)\mapsto 
\int_0^t[g(\bar x(s))u(s)]_j\,ds
\]
is a linear bounded functional for each $t\geqslant0$ since 
\begin{align*}
\left| \int_0^t[g(\bar x(s))u(s)]_j\,ds \right|
&\leqslant \int_0^t\|g(\bar x(s))\|\,|u(s)|\,ds\\
&\leqslant \left(\int_0^t\|g(\bar x(s))\|^2\,ds\right)^{1/2}\|u\|_{L^2((0,\infty);\mathbb{R}^m)}. 
\end{align*}
Therefore, we have 
\begin{equation}
\lim_{m\to\infty}\int_0^tg(\bar x(s))u_m(s)\,ds=\int_0^tg(\bar x(s))\bar u(s)\,ds\label{eqn:lim_gbar_um}
\end{equation}
from the weak convergence of $u_m$ in $L^2((0,\infty);\mathbb{R}^m)$. 
From (\ref{ineq:conv_gmum}) and (\ref{eqn:lim_gbar_um}), 
\begin{align*}
\lim_{m\to\infty}\int_0^tg(x_m(s))u_m(s)\,ds 
&=\lim_{m\to\infty} \left[\int_0^t g(\bar x(s))u_m(s)\,ds +\int_0^t(g(x_m(s))-g(\bar x(s)))u_m(s)\,ds\right]\\
&=\int_0^tg(\bar x(s))\bar u(s)\,ds. 
\end{align*}

\noindent
\underline{Step 2.} 
For a sufficiently large $m$, 
\begin{align*}
1+\inf_{L^2((0,\infty);\mathbb{R}^m)}J &\geqslant \int_0^\infty \frac{1}{2}|u_m(t)|^2+h(x_m(t))\,dt\\
&\geqslant \int_0^\infty h(x_m(t))\,dt=\|h(x_m)^{1/2}\|_{L^2((0,\infty);\mathbb{R})}
\end{align*}
and $\{h(x_m)^{1/2}\}$ is a bounded set in $L^2((0,\infty);\mathbb{R})$. Replacing $\int_0^\infty\,dt$ with $\int_0^T\,dt$, it is also a bounded set in $L^2((0,T);\mathbb{R})$ for all $T>0$. By Banach-Alaoglu Theorem and the diagonal argument, up to subsequence, we have 
\begin{gather*}
h(x_m)^{1/2}\to l \quad \text{weakly in }L^2((0,\infty);\mathbb{R}),\\
h(x_m)^{1/2}\to l \quad \text{weakly in }L^2((0,T);\mathbb{R}) \quad \text{for all }T>0,
\end{gather*}
as $m\to\infty$ for some $l\in L^2((0,\infty);\mathbb{R})$. On the other hand, 
\[
\int_0^T|h(x_m(t))-h(\bar x(t))|^2\,dt\to0 \quad\text{as }m\to\infty \text{ for all }T>0, 
\]
implying that
\[
h(x_m)^{1/2}\to h(\bar x)^{1/2} \quad \text{strongly in }L^2((0,T);\mathbb{R}) \quad \text{for all }T>0.
\]
From the uniqueness of weak limit, we have $h(\bar x)=l$ on $[0,T]$ for all $T>0$. Thus, we have shown that 
\begin{equation}
h(x_m)^{1/2}\to h(\bar x)^{1/2} \quad \text{weakly in }L^2((0,\infty);\mathbb{R}). \label{eqn:weaklimit_h}
\end{equation}
\noindent
\underline{Step 3.} 
From the weak convergence of $u_m$ to $\bar u$ in $L^2((0,\infty);\mathbb{R}^m)$, $h(x_m)^{1/2}$ to $h(\bar x)^{1/2}$ in $L^2((0,\infty);\mathbb{R})$
and lower semi-continuity of norm for weak topology, 
\begin{gather*}
\|\bar u\|_{L^2((0,\infty);\mathbb{R}^m)}\leqslant \liminf_{m\to\infty}\|u_m\|_{L^2((0,\infty);\mathbb{R}^m)},\\
\|h(\bar x)^{1/2}\|_{L^2((0,\infty);\mathbb{R})}\leqslant
    \liminf_{m\to\infty}\|h(x_m)^{1/2}\|_{L^2((0,\infty);\mathbb{R})},
\end{gather*}
and it holds that 
\begin{align*}
    J(\bar u) =\int_0^\infty \frac{1}{2}|\bar u(t)|^2+h(\bar x(t))\,dt &\leqslant \frac{1}{2}\liminf_{m\to\infty}\|u_m\|_{L^2((0,\infty);\mathbb{R}^m)}
    +\liminf_{m\to\infty}\int_0^\infty h(x_m(t))\,dt\\
    &= \liminf_{m\to\infty}J(u_m)=\inf_{L^2((0,\infty);\mathbb{R}^m)}J.
\end{align*}
This proves that $\bar u$ is an optimal control. 
\end{proof}

\subsection{Exponentially stable case}\label{sctn:exp_stable}
In this section, we consider the case where the free dynamics $\dot x =f(x)$ in (\ref{eqn:n_sys}) is globally exponentially stable. That is, there exist constants \clred{$\mu>0$ and $K>0$ that are independent of $x_0$} such that the following estimate for the corresponding solution $x(t,x_0)$ holds 
\[
|x(t,x_0)|\leqslant K\clred{|x_0|}e^{-\mu t} \text{ for }t\geqslant0, \ x_0\in \mathbb{R}^n.
\]
\begin{theorem}\label{prop:exp_stable} 
Assume that $g(x)$ is bounded in $\mathbb{R}^n$ and free dynamics $\dot x=f(x)$ is globally exponentially stable. Assume also that $Df(x)$ is bounded in $\mathbb{R}^n$.\footnote{
These assumptions are necessary to have \clred{Lyapunov function} $V(x)$ defined on $\mathbb{R}^n$.} 
Then, for each $x_0\in \mathbb{R}^n$, there exists an optimal control \clred{$\bar u$} for (\ref{eqn:n_sys})-(\ref{eqn:cost}) \clred{and it holds that $x_{\bar u}(t)\to0$ as $t\to\infty$}. 
\end{theorem}
\begin{proof}
We prove that the assumptions in Proposition~\ref{prop:minimizer} are satisfied. Namely, we show that for any $u\in L^2((0,\infty);\mathbb{R}^m)$, the corresponding solution belongs to $H^1((0,\infty);\mathbb{R}^n)$ and hence, for any bounded set $\mathscr{U}\subset L^2((0,\infty);\mathbb{R}^m)$, the corresponding set $\{\clred{x_u}\,|\,u\in\mathscr{U}\}$ is bounded in $C^0([0,T];\mathbb{R}^n)$ for all $T>0$. \\
\underline{$x\in L^2((0,\infty);\mathbb{R}^n)$:} 
From the global exponential stability of $\dot x=f(x)$ and the boundedness of $Df(x)$, there exist a $C^1$ function $V:\mathbb{R}^n\to\mathbb{R}$\clred{,} positive constants $c_1$, $c_2$, $c_3$ and $c_4$ such that for all $x\in \mathbb{R}^n$
\begin{enumerate}[i)]
    \item 
    $ c_1|x|^2\leqslant V(x)\leqslant c_2|x|^2$ 
    \item
    $DV(x)f(x)\leqslant -c_3|x|^2$
    \item
    $|DV(x)|\leqslant c_4|x|$
\end{enumerate}
(see, e.g., page 180 of \cite{Khalil:92:NS}). \clred{Take an $u\in L^2((0,\infty);\mathbb{R}^m)$ and let $x(t)=x_u(t)$.} 
Using the above inequalities and the boundedness of $\|g(x(t))\|$ on the trajectory, one can derive
\[
\frac{d}{dt}V(x(t))\leqslant -cV(x(t))+c'|u(t)|^2, \ t\geqslant0, 
\]
for some positive constants $c$, $c'$ (which depend on $g$). 
Applying Gronwall's inequality and the quadratic estimates on $V$, it follows that
\[
|x(t)|^2\leqslant c|x_0|^2 e^{-c't}+K\int_0^te^{-c'(t-s)}|u(s)|^2\,ds, \ t\geqslant0
\]
for some positive constants $c$, $c'$ and $K$ (that depend on $U$ and $g$).  
Since 
\[
\int_0^\infty\left( \int_0^t e^{-c'(t-s)}|u(s)|^2\,ds \right)\,dt
=\frac{1}{c'}\left( \int_0^\infty |u(t)|^2\,dt - \lim_{t\to\infty}e^{-c't}\int_0^t e^{c's}|u(s)|^2\,ds \right)
\]
and the second term on the right is 0, we have
\begin{equation}
\int_0^\infty|x(t)|^2\,dt\leqslant \frac{c}{c'}|x_0|^2 +\frac{1}{c'}\|u\|_{L^2((0,\infty);\mathbb{R}^m)}^2.\label{ineq:x_in_L2} 
\end{equation}
\underline{$\dot x\in L^2((0,\infty);\mathbb{R}^n)$:} 
Applying Lemma~\ref{lemma:h_convergence} with $H(x)=|x|^2$, we know that $x(t)$ is bounded for $t\geqslant0$. Therefore, 
\begin{align*}
\int_0^\infty|\dot{x}(t)|^2\,dt &= \int_0^\infty |f(x(t))+g(x(t))u(t)|^2\,dt\\ 
&\leqslant 2\int_0^\infty |f(x(t))|^2+\|g(x(t))\|^2\,|u(t)|^2\,dt\\
&\leqslant 2M^2\int_0^\infty |x(t)|^2\,dt+\sup_{t\geqslant0}\|g(x(t))\|^2
    \int_0^\infty|u(t)|^2\,dt\\
&\leqslant 2M^2\|x\|_{L^2((0,\infty);\mathbb{R}^n)}^2+\sup_{t\geqslant0}\|g(x(t))\|^2 \|u\|_{L^2((0,\infty);\mathbb{R}^m)}^2,
\end{align*}
where $M>0$ is a constant satisfying $|f(x)|\leqslant M|x|$ in a bounded set that contains $\{x(t)\,|\,t\geqslant0\}$. This and (\ref{ineq:x_in_L2}) show that $x\in H^1((0,\infty);\mathbb{R}^n)$ with 
\begin{equation}
    \|x\|_{H^1((0,\infty);\mathbb{R}^n)}\leqslant c_1+c_2\|u\|_{L^2((0,\infty);\mathbb{R}^m)}\label{ineq:x_h1norm_u}
\end{equation}
for some positive constants $c_1$, $c_2$. 

Now, we have for $t\in[0,T]$, where $T>0$ is arbitrary, 
\begin{align*}
|x(t)|&\leqslant \int_0^t|\dot{x}(s)|\,ds +|x_0|\\
&\leqslant \sqrt{t}\|\dot{x}\|_{L^2((0,\infty);\mathbb{R}^n)}+|x_0|\\
&\leqslant \sqrt{T}\|x\|_{H^1((0,\infty);\mathbb{R}^n)}+|x_0|
\leqslant C_T+C'_T\|u\|_{L^2((0,\infty);\mathbb{R}^m)}, \qquad \text{from (\ref{ineq:x_h1norm_u})}
\end{align*}
for some positive constants $C_T$, $C'_T$ that depend on $T$. This proves that the assumptions in Proposition~\ref{prop:minimizer} are satisfied and an optimal control exists. \clred{The last assertion of the theorem follows from the same estimate $\int_0^\infty|x_{\bar u}(t)|^2\,dt<\infty$ and Lemma~\ref{lemma:h_convergence} with $H(x)=|x|^2$.}
\end{proof}

\subsection{Exponentially stabilizable case}
In this subsection, we consider (\ref{eqn:n_sys}) for a more general case where there exists a $C^1$ exponentially stabilizing feedback control $u=k(x)$ with $k(0)=0$ for initial points in \clred{a neighborhood of $x=0$}. 
Let us recall zero-state detectability for nonlinear systems, which is often imposed in OCPs. 
\begin{definition} System $\dot x=f(x)$ and output $y=h(x)$ (or, simply the pair $(f,h)$) is said to be {\em zero-state detectable for a neighborhood $U\subset \mathbb{R}^n$ of the origin} if the following holds. If a solution $x(t)$ with $x(0)\in U$ satisfies $h(x(t))=0$ for $t\geqslant 0$, then $x(t)\to0$ as $t\to\infty$. 
\end{definition}
\begin{proposition}\label{prop:detect}
Suppose that $h$ is a $C^1$ function which is coercive, namely, $h(x)\to\infty$ if $|x|\to\infty$. Take positive constants $R$, $\mu$ such that $h(x)>\mu$ if $|x|>R$. Assume also that the pair $(f,h)$ is zero-state detectable for an open set containing $|x|\leqslant R$. Let $\delta(t)\in L^2((0,\infty);\mathbb{R}^n$). 
Assume finally that solution 
$x_\delta(t)$ for $\dot x=f(x)+\delta(t)$ satisfies $h(x_\delta(t))\to0$ as $t\to\infty$. 
Then, $x_\delta(t)\to0$ as $t\to\infty$. 
\end{proposition}
\begin{proof}
From the coercivity of $h$, $\{x_\delta(t)|\,t\geqslant 0\}$ is bounded in $\mathbb{R}^n$. Using Bolzano-Weierstrass Theorem, there exists a subsequence $\{x_\delta(t_n)\}$, $t_n\to\infty$ ($n\to\infty$), such that $x_\delta(t_n)\to\bar x$ as $t_n\to\infty$. Define $\delta_n(t):=\delta(t_n+t)$, $f_n(t,x):=f(x)+\delta_n(t)$ and $\xi_n:=x_\delta(t_n)$. Note that $\xi_n\to\bar x$ as $n\to\infty$. We consider an initial value problem
\[
\dot x =f_n(t,x),\quad x(0)=\xi_n.
\]
Let $\varphi_n(t)$ denote the solution of the above problem. One then notices that $\varphi_n(t)=x_\delta(t_n+t)$ and therefore $\varphi_n(t)$ is bounded for $t\geqslant0$ \clred{
uniformly in $n$}. 
Also, let $\varphi(t)$ be the solution on $[0,t_1]$ for 
\[
\dot x=f(x), \quad x(0)=\bar x.
\]
Write $f_0(t,x):=f(x)$. Then, for any bounded set $\bar D\in\mathbb{R}^n$, there exists a constant $M>0$ such that for any $x$, $y\in \bar D$ it follows that
\begin{gather*}
|f_n(t,x)-f_0(t,x)|\leqslant |\delta_n(t)|,\\
\int_0^t|\delta_n(s)|\,ds=\int_{t_n}^{t+t_n}|\delta(s)|\,ds\to0, \text{ as }n\to \infty \text{ for }t>0,\\
\int_{t_1}^{t_2}|\delta_n(t)|\,dt \leqslant \sqrt{t_2-t_1}\| \delta \|_{L^2((0,\infty);\mathbb{R}^n)} \text{ for }0\leqslant t_1\leqslant t_2,\\
|f_n(t,x)-f_n(t,y)|=|f(x)-f(y)|\leqslant M|x-y|. 
\end{gather*}
Now all the assumptions in Proposition~\ref{prop:hale_lemma} are satisfied. Therefore, $\varphi_n\to\varphi$ uniformly on $[0,t_1]$ as $n\to\infty$. However, from the boundedness of $\varphi_n$, $\varphi$ is defined on $[0,\infty)$ and also bounded for $t\geqslant0$. Take a constant $M_h>0$ \clred{that is independent of $n$} such that in a bounded set that contains $\varphi_n(t)$ and $\varphi(t)$ for $t\geqslant0$, 
\[
|h(x)-h(y)|\leqslant M_h|x-y|
\]
holds for all $x$, $y$ in the bounded set. Then we have 
\[
0\leqslant h(\varphi(t))\leqslant M_h|\varphi_n(t)-\varphi(t)|+h(x_\delta(t_n+t))
\]
for $t\geqslant0$. Taking the limit $n\to\infty$, we obtain 
\[
h(\varphi(t))=0 \quad \text{for}\quad t\geqslant0.
\]
We have then $|\varphi(t)|\leqslant R$ for $t\geqslant0$. From the detectability of $(f,h)$ for the open set containing $|x|\leqslant R$, we have $\varphi(t)\to0$ as $t\to\infty$ and therefore $x_\delta(t)\to0$ as $t\to\infty$ since $\varphi_n\to\varphi$ uniformly on $[0,T]$ as $n\to\infty$ for any $T>0$. 
\end{proof}
%
\begin{theorem} \label{thm:minimizer_2} 
Assume that system (\ref{eqn:n_sys}) is exponentially stabilizable by a $C^1$ feedback control for \clred{an open set} $U\subset\mathbb{R}^n$ and that $h$ is a $C^1$ function of $x\in\mathbb{R}^n$. Assume also that the pair $(f,h)$ is zero-state detectable for an open set containing $|x|\leqslant\rho$ for some $\rho>0$. 
Assume finally that there exist positive constants $p$, $c_f$, $c_g$ and a constant $0\leqslant\theta<1$ such that 
\begin{subequations}
\begin{align}
|f(x)|&\leqslant c_f|x|^{p+\theta} \label{ineq:groth_f}\\
\|g(x)\|&\leqslant c_g|x|^{p/2+\theta}\label{ineq:groth_g}
\end{align}
\end{subequations}
for sufficiently large $x\in\mathbb{R}^n$ and that there exists a positive constants $c_h$ such that 
\begin{equation}
h(x) \geqslant c_h|x|^p\label{ineq:groth_h}
\end{equation}
for $x$ with $|x|>\rho$. Then, for $x_0\in U$, OCP (\ref{eqn:n_sys})-(\ref{eqn:cost}) has an optimal control $\bar u\in L^2((0,\infty);\mathbb{R}^m)$ \clred{and it holds that $x_{\bar u}(t)\to0$ as $t\to\infty$}. 
\end{theorem}
\begin{proof}
Let $k(x)$ be a $C^1$ exponentially stabilizing feedback and let $u_e(t)=k(x(t))$. Then $J_e:=J(u_e)<+\infty$ since the closed loop solution satisfies the exponential decay and around the origin we have estimates on $k$ and $h$; $|k(x)|\leqslant M_k|x|$, $h(x)\leqslant M_h|x|$ for some constants $M_h,M_k>0$. There exists a minimizing sequence $\{u_m\}\subset L^2((0,\infty);\mathbb{R}^m)$, which is bounded in $L^2((0,\infty);\mathbb{R}^m)$ as in the proof of Proposition~\ref{prop:minimizer}, such that
\[
\lim_{m\to\infty}J(u_m)=\inf_{L^2((0,\infty);\mathbb{R}^m)}J, 
\quad J(u_m) <J_e\  (m\in\mathbb{N}). 
\]
Let \clred{$x_m=x_{u_m}$}. Then, we have $\int_0^\infty h(x_m(t))\,dt<\infty$. From Lemma~\ref{lemma:h_convergence} with $H(x)=h(x)$, $h(x_m(t))\to0$ as $t\to\infty$ and $x_m(t)$ is bounded for $t\geqslant0$.  
With the correspondence 
\[
\delta(t)\longleftrightarrow g(x_m(t))u_m(t), \quad 
x_\delta \longleftrightarrow x_m,\quad R\longleftrightarrow\rho,
\]
in Proposition~\ref{prop:detect}, we have $x_m\in C^0([0,\infty);\mathbb{R}^n)$, $x_m(t)\to0$ as $t\to\infty$. 

We next prove that $x_m(t)$ is uniformly bounded. Let $L_m:= \sup_{t\geqslant0}|x_m(t)|$ and we prove that $\sup_{m\in\mathbb{N}}L_m<\infty$. Define $I_m\subset\mathbb{R}$ by 
\[
I_m:=\{t\geqslant0\,|\, |x_m(t)|\geqslant L_m/2\}. 
\]
Assume that $\sup_{m\in\mathbb{N}}L_m=\infty$ and take $\{m_k\}\subset \mathbb{N}$ such that $L_{m_k}\to\infty$ as $m_k\to\infty$. Then, for $m_k$ sufficiently large, 
\begin{align*}
J_e&\geqslant \int_0^\infty h(x_{m_k}(t))\,dt\\
&\geqslant \int_{I_{m_k}}h(x_{m_k}(t))\,dt\\
&\geqslant c_h\int_{I_{m_k}} |x_{m_k}(t)|^p\,dt\qquad \text{by (\ref{ineq:groth_h})}\\
&\geqslant c_h\left(\frac{L_{m_k}}{2}\right)^p|I_{m_k}|,
\end{align*}
where $|\cdot|$ denotes length (Lebesgue measure) of interval. Thus we have 
\begin{equation}
|I_{m_k}|\leqslant C {L_{m_k}}^{-p},\label{ineq:I_m_length}
\end{equation}
where $C>0$ is independent of $m_k$. Next we compute the length of trajectory connecting spheres $|x|=L_{m_k}$ and $|x|=L_{m_k}/2$, which is 
\begin{align*}
\int_{I_{m_k}}|\dot{x}_{m_k}(t)|\,dt &\leqslant 
    \int_{I_{m_k}}|f(x_{m_k}(t))|+\|g(x_{m_k})\|\,|u_{m_k}(t)|\,dt\\
    &\leqslant |I_{m_k}|\sup_{L_{m_k}/2\leqslant |x|\leqslant L_{m_k}}|f(x)|
    +\sup_{L_{m_k}/2\leqslant |x|\leqslant L_{m_k}}\|g(x)\|\int_{I_{m_k}}|u_{m_k}(t)|\,dt\\
&\leqslant |I_{m_k}|\sup_{L_{m_k}/2\leqslant |x|\leqslant L_{m_k}}|f(x)|\\
&\qquad    +\sqrt{|I_{m_k}|}\sup_{L_{m_k}/2\leqslant |x|\leqslant L_{m_k}}\|g(x)\|\sup_{m\in\mathbb{N}}\|u_m\|_{L^2((0,\infty);\mathbb{R}^m)}\\
&\leqslant c_1|I_{m_k}|{L_{m_k}}^{p+\theta}+c_2\sqrt{|I_{m_k}|}{L_{m_k}}^{p/2+\theta},
\qquad \text{by (\ref{ineq:groth_f}) and (\ref{ineq:groth_g})}
\end{align*}
where $c_1$, $c_2$ are positive constants independent of $m_k$. Then, from (\ref{ineq:I_m_length}) it follows that 
\[
\int_{I_{m_k}}|\dot{x}_{m_k}(t)|\,dt\leqslant C{L_{m_k}}^{\theta}, 
\]
where $C>0$ is a constant independent of $m_k$. However, the left side of above grows, at least, as $O(|L_{m_k}|)$, which is a contradiction. 

The rest of the proof is the same as Proposition~\ref{prop:minimizer}. \clred{The last assertion follows using Proposition~\ref{prop:detect} with $\delta(t)=g(x_{\bar u}(t))\bar{u}(t)$.} 
\end{proof}

\begin{remark}
Although exponential stability of free dynamics $\dot x=f(x)$ implies exponential stabilizability, Theorem~\ref{thm:minimizer_2} does not include Theorem~\ref{prop:exp_stable} since in Theorem~\ref{prop:exp_stable} $h\equiv 0$ is allowed. 
\end{remark}
\clred{
\begin{remark}
Theorem~\ref{thm:minimizer_2} cannot handle the case $h(x)=|Cx|^2$ with a constant matrix $C\in\mathbb{R}^{r\times n}$. This is due to the coercivity of $h$ used in the proof. 
\end{remark}
}
\section{Stable manifold analysis of associated Hamiltonian system}\label{sctn:stabl_manifold}
In this section, we additionally assume that $h(x)$ is $C^2$ and $Dh(0)=0$.
The second problem treated in the paper is to give estimates on a stable manifold of the associated Hamiltonian system for OCP (\ref{eqn:n_sys})-(\ref{eqn:cost});  
\begin{subequations}\label{eqn:hsys_all}
\begin{align}
    \dot{x}&=f(x)-g(x)g(x)^\top p \label{eqn:hsys1}\\
    \dot{p}&=-Df(x)^\top p +\frac{1}{2}D(p^\top g(x)g(x)^\top p)^\top -Dh(x)^\top \label{eqn:hsys2}.
\end{align}
\end{subequations}
\clred{The Hamiltonian system (\ref{eqn:hsys_all}) appears in optimal control theory in two ways. The first is from Pntryagin's minimum principle, where (\ref{eqn:hsys_all}) takes the form of 
\begin{gather*}
\dot x=\frac{\partial H_D}{\partial p}(x(t),u^\ast (t), p(t))^\top,\quad 
\dot p=-\frac{\partial H_D}{\partial x}(x(t),u^\ast (t), p(t))^\top,\\
u^\ast(t) = \min_{u\in\mathbb{R}^m}H_D(x(t),u, p(t)),
\end{gather*}
for $H_D(x,u,p)=p^\top f(x)+p^\top g(x)u+\frac{1}{2}|u|^2+h(x)$. The minimization over $u\in\mathbb{R}^m$ can be expressed explicitly and (\ref{eqn:hsys_all}) is obtained. The second is from the method of characteristics for the HJBE  
\[
H(x,\partial V/\partial x)=0
\]
obtained by the Dynamic Programming, where $H(x,p)=\min_{u\in\mathbb{R}^n}H_D(x,u,p)$. 
}

Let us explain the motivation of the second problem in the paper using a simple example 
\[
\dot x=-x+x^2+u,\quad J=\int_0^\infty u(t)^2/2\,dt. 
\]
This system is globally exponentially stabilizable by $u=-x^2$ and Hamiltonian system is 
\begin{equation}
\dot x= -x+x^2-p,\quad \dot p= p-2xp.\label{eqn:ex_hamsys} 
\end{equation}
The optimal control is $u=0$ for all $x(0)\in\mathbb{R}$ and the stable manifold for the Hamiltonian system (\ref{eqn:ex_hamsys}) exists only in $x<1$. In (\ref{eqn:ex_hamsys}), there are two \clred{equilibria}; $(x,p)=(0,0), (1,0)$ and heteroclinic orbits connecting them exist. To guarantee the global existence of stable manifold, one realizes that the detectability condition is necessary. 

\clred{It should be also noted that the asymptotic behavior of $p(t)$ is related with the transversality condition in the minimum principle, but, in infinite horizon OCPs, it does not hold in general that $p(\infty)=0$. For this issue, we refer to \cite{Halkin:74:econometrica}. The paper shows a counter example using an example with singular costate (which does not happen in our setting). In the following, we show that in (\ref{eqn:hsys_all}) $p(t)$ satisfies $p(\infty)=0$.}

\begin{proposition}\label{prop:p_exists_stable} Suppose that $h(x)$ is $C^2$ and $Dh(0)=0$. 
\begin{enumerate}[(i)]
    \item 
Assume that hypotheses in Theorem~\ref{prop:exp_stable} are satisfied. Then, for any $x_0\in\mathbb{R}^n$, Hamiltonian system (\ref{eqn:hsys_all}) admits a unique solution $(x(t),p(t))$ defined on $[0,\infty)$ satisfying $x(0)=x_0$ and $(x(t),p(t))\to0$ as $t\to\infty$.
\item
\clred{Assume that all the hypotheses in Theorem~\ref{thm:minimizer_2} are satisfied. Assume additionally that the linear part of $(h(x), f(x)))$ is detectable.} \clred{Then} for any $x_0\in U$, \clred{where $U$ is the open set for the exponential stabilization}, Hamiltonian system (\ref{eqn:hsys_all}) admits a unique solution $(x(t),p(t))$ defined on $[0,\infty)$ satisfying $x(0)=x_0$ and $(x(t),p(t))\to0$ as $t\to\infty$.
\end{enumerate}
\end{proposition}
\begin{proof}
Let us write 
\begin{gather*}
f(x)=Ax+\varphi(x), \ \varphi(x)=O(|x|^2), \quad g(x)=B+\tilde{g}(x),\ \tilde{g}(x)=O(|x|)\\
h(x)=\frac{1}{2}|Cx|^2+\tilde{h}(x),\ \tilde{h}(x)= O(|x|^3),\ |x|\to0, 
\end{gather*}
where $\varphi$, $\tilde g(x)$ and $\tilde h(x)$ are higher order $C^2$ maps. 
\clred{The proofs for (i) and (ii) are almost the same since $A$ is a Hurwitz matrix in (i) and the exponential stabilizability by $C^1$ feedback in (ii) implies that $(A,B)$ is stabilizable. One additionally needs the detectability of $(C,A)$ for (ii).}

The Hamiltonian system (\ref{eqn:hsys_all}) can be written as 
\begin{gather}
\frac{d}{dt}\bmat{x\\p}=\mathrm{Ham}\bmat{x\\p}
    +\bmat{\varphi(x) -\Phi(x)p\\ -D\varphi(x)^\top p +\frac{1}{2}D(p^\top g(x)g(x)^\top p)^\top -D\tilde{h}(x)^\top}\label{eqn:ham_all}\\
\intertext{where}
\mathrm{Ham}=\bmat{A& -BB^\top\\-C^\top C&-A^\top}, \quad 
\Phi(x)=B\tilde{g}(x)^\top +\tilde{g}(x)B^\top+\tilde{g}(x)\tilde{g}(x)^\top.\nonumber
\end{gather}
\clred{
(\ref{eqn:ham_all}) is written as
\[
\frac{d}{dt}\bmat{x\\p}=\mathrm{Ham}\bmat{x\\p}
    +\bmat{\gamma_1(x,p)\\\gamma_2(x,p)}
\]
with appropriately defined $\gamma_1(x,p)$, $\gamma_2(x,p)$. Note that $\gamma_j(x,p)=O((|x|+|p|)^2)$, $|x|+|p|\to0$, $j=1,2$.
}

\clred{Since $(A,B)$ is stabilizable and $(C,A)$ is detectable}, the following Riccati equation has a solution $P_1\geqslant0$
\[
PA+A^\top P-PBB^\top P+C^\top C=0
\]
with $\mathrm{Re}\,\lambda(A-BB^\top P_1)<0$. Also, take a solution $P_2\leqslant0$ for the following Lyapunov equation
\[
P(A-BB^\top P_1)^\top +(A-BB^\top P_1)P=BB^\top.
\]
Using a symplectic transformation (see, e.g., \cite{Lukes:69:sicon,Sakamoto:02:sicon} for detail)
\[
L=\bmat{I&P_2\\P_1& I+P_1P_2},\quad L^{-1}=\bmat{I+P_2P_1&-P_2\\-P_1&I},
\]
the linear part $\mathrm{Ham}$ is block-diagonalized 
\[
L^{-1}\mathrm{Ham}L=\bmat{A-BB^\top P_1&0\\0&-(A-BB^\top P_1)^\top}.
\]
Let us introduce new coordinates $\xi$, $\eta$ by 
\begin{equation}
\bmat{x\\p}=L\bmat{\xi\\\eta}=\bmat{\xi+P_2\eta\\P_1\xi+(P_1P_2+I)\eta}
=\bmat{\xi+P_2\eta\\P_1 x+ \eta}.\label{eqn:xi_eta}
\end{equation}
The Hamiltonian system is then written, in the coordinates $(\xi, \eta)$, as
\begin{equation}
\frac{d}{dt}\bmat{\xi\\ \eta}=\bmat{(A-BB^\top P_1)&0\\0&-(A-BB^\top P_1)^\top}\bmat{\xi\\\eta}
    +\bmat{\nu_1(\xi,\eta)\\\nu_2(\xi,\eta)},\label{eqn:ham_all_xi_eta}
\end{equation}
where $\nu_j(\xi,\eta)=O((|\xi|+|\eta|)^2)$, $|\xi|+|\eta|\to0$, $j=1,2$. It is known (see, e.g., Page~107 of \cite{Chow:82_96:MBT}) that in (\ref{eqn:ham_all_xi_eta}), there exists a unique $C^1$ stable manifold $\eta=\theta(\xi)$ in a neighborhood of $(\xi,\eta)=(0,0)$ satisfying $\theta(0)=0$, $D\theta(0)=0$ such that the solution $\xi(t)$, $\eta(t)$ satisfy $\xi(t)=\theta(\eta(t))$ for $t\geqslant0$ and $\xi(t)$, $\eta(t)\to0$ provided that $\eta(0)=\theta(\xi(0))$%
\footnote{For a dynamical system whose linearization at the equilibrium (denoted $0$) has no eigenvalues on the imaginary axis, a stable manifold exists around the equilibrium, which is defined as
\[
\bigcup_{t\geqslant0}\varphi_{-t}(W_{\mathrm{loc}}^s(0)),
\]
where $\varphi_t(z)$ is the flow of the dynamical system starting from $z$ at $t=0$ and $W_{\mathrm{loc}}^s(0)$ is the set of points near the equilibrium from which the flow is bounded for $t\geqslant0$ (local stable manifold).
}. 

\clred{Take an $x_0$ for which an optimal control $u^\ast$ exists by either Theorem~\ref{prop:exp_stable} or Theorem~\ref{thm:minimizer_2} and let $x^\ast$ be the optimal trajectory, which satisfies $x^\ast(t)\to0$ as $t\to\infty$. For arbitrary large $t_1\geqslant0$, there exists a $\xi_1\in\mathbb{R}^n$ such that $x^\ast (t_1)=\xi_1+P_2\theta(\xi_1)$. This is shown from  
\[
\left.\frac{\partial}{\partial \xi}(\xi+P_2\theta(\xi))\right|_{\xi=0}=I,
\]
which uses $\theta(0)=0$, $D\theta(0)=0$ and from the Implicit Function Theorem if $|x^\ast (t_1)|$ is sufficiently small. 
Let $(\xi(t),\eta(t))$ be the solution for (\ref{eqn:ham_all_xi_eta}) with $\eta(t_1)=\theta(\xi_1)$ and define $(x(t),p(t))$ by (\ref{eqn:xi_eta}). Then, $(x(t),p(t))$ satisfy (\ref{eqn:hsys_all}) and therefore $x(t)$ is an optimal trajectory for $t\geqslant t_1$ (this is guaranteed, for instance, by the result in \cite{Lukes:69:sicon}). However, from the principle of optimality, it follows that $x^\ast (t)=x(t)$ for $t\geqslant t_1$ since $x(t_1)=x^\ast (t_1)$. This shows that $p(t)$ defined above satisfies (\ref{eqn:hsys_all}) together with $x^\ast (t)$ and $p(t)\to0$ as $t\to\infty$.

To prove that $p(t)$ satisfying (\ref{eqn:hsys_all}) exists on $[0,\infty)$, we show that $p(t)$ can be extended to $[0,t_1]$. To do this, we equivalently consider (\ref{eqn:hsys2}) as the second equation from the necessary condition of optimality; 
\[\clred{
\dot p = -[Df(x^\ast (t))^\top +L(x^\ast (t),u^\ast (t))]p -Dh(x^\ast (t))^\top,
}
\]
where
\[
L(x,u)= \sum_{j=1}^m u_jDg_j(x)^\top,\quad 
g(x)=\bmat{g_1(x)&\cdots&g_m(x)}.
\]
}%
This is a linear equation for $p$ and solution exists in the sense of Carath\'eodory for all $t\in[0,t_1]$. 
\end{proof}
\clred{
Under the assumptions in Proposition~\ref{prop:p_exists_stable}, the linearization of (\ref{eqn:hsys_all}) at the origin has no eigenvalues on the imaginary axis, and therefore, a stable manifold $\Lambda_S$ at the origin exists. The following theorem is a restatement of Proposition~\ref{prop:p_exists_stable} in terms of the stable manifold, which simply says that the stable manifold of (\ref{eqn:hsys_all}), when it is projected to the base $x$-space, covers the region of stability or stabilizability. 
}
Let $\pi:\mathbb{R}^n\times\mathbb{R}^n\to\mathbb{R}^n$ be the canonical projection; $\pi(x,p)=x$. 
\begin{theorem}\label{thm:smani}
Suppose that $h(x)$ is $C^2$ and $Dh(0)=0$.
\begin{enumerate}[(i)]
    \item 
Assume that hypotheses in Theorem~\ref{prop:exp_stable} are satisfied. Then, for the Hamiltonian system (\ref{eqn:hsys_all}), the stable manifold $\Lambda_S$ \clred{satisfies} $ \pi(\Lambda_S)=\mathbb{R}^n$.
\item
\clred{Assume that all the hypotheses in Theorem~\ref{thm:minimizer_2} are satisfied. Assume additionally that the linear part of $(h(x), f(x)))$ is detectable.} 
Then, for the Hamiltonian system (\ref{eqn:hsys_all}), the stable manifold $\Lambda_S$ \clred{satisfies} $ U\subset \pi(\Lambda_S)$, \clred{where $U$ is the open set for the exponential stabilization}.
\end{enumerate}
\end{theorem}

\section{Applications}\label{sctn:applications}
This section first provides a tool to weaken the growth conditions in Theorem~\ref{thm:minimizer_2}. Let $x=(x_1,x_2)$ with $x_1\in\mathbb{R}^{n_1}$, $x_2\in\mathbb{R}^{n_2}$, $n_1+n_2=n$. Rewrite (\ref{eqn:n_sys}) as
\[
\frac{d}{dt}\bmat{x_1\\x_2}=\bmat{f_1(x_1,x_2)\\f_2(x_1,x_2)}+\bmat{g_1(x_1,x_2)\\g_2(x_1,x_2)}u,
\]
where $f_j:\mathbb{R}^n\to\mathbb{R}^{n_j}$, $g_j:\mathbb{R}^n\to\mathbb{R}^{n_j\times m}$, $j=1,2$. 
Let $\varphi_R:\mathbb{R}^{n_2}\to\mathbb{R}$ be a $C^\infty$ cutoff function such that $\varphi_R(x_2)=1$ for $|x_2|<R$ and $\varphi_R(x_2)=0$ for $|x_2|\geqslant R+1$. Define $\tilde{f}_R(x_1,x_2):=f(x_1,\varphi_R(x_2)x_2)$ and $\tilde{g}_R(x_1,x_2):=g(x_1,\varphi_R(x_2)x_2)$. 
\begin{proposition}\label{prop:minimizer_2_appl}Assume that (\ref{eqn:n_sys}) is $C^1$-exponentially stabilizable for \clred{an open set} $U\subset\mathbb{R}^n$ and that $h$, which is $C^1$, satisfies (\ref{ineq:groth_h}) for some positive constants $p$, $c_h$, $\rho$ and for $|x|>\rho$. Suppose that $(f,h)$ is zero-state detectable for an open set containing $|x|\leqslant\rho$.   
Assume finally the following.  
\begin{enumerate}[(i)]
    \item  For any $R>0$, there exist positive constants $c_f=c_f(R)$, $c_g=c_g(R)$ and $0\leqslant\theta<1$ which is independent of $R$ such that $\tilde{f}_R$ and $\tilde{g}_R$ satisfy (\ref{ineq:groth_f}) and (\ref{ineq:groth_g}), respectively, for sufficiently large $x\in\mathbb{R}^n$. 
    \item There exist constants $c_{f2}>0$, $c_{g2}>0$, $0\leqslant\theta_2<1$ such that 
    \begin{subequations}\label{ineq:groth_all_app}
\begin{align}
|{f}_2(x_1,x_2)|&\leqslant c_{f2}|x_2|^{p+\theta_2} \label{ineq:groth_f_appl}\\
\|{g}_2(x_1,x_2)\|&\leqslant c_{g2}|x_2|^{p/2+\theta_2}\label{ineq:groth_g_appl}
\end{align}
    \end{subequations}
    for all $(x_1,x_2)\in\mathbb{R}^n$. 
\end{enumerate}
Then, an optimal control $u_R\in L^2((0,\infty);\mathbb{R}^m)$ exists for an OCP $\dot{x}=\tilde{f}_R(x)+\tilde{g}_R(x)u$, $x(0)=x_0\in U$ and cost functional (\ref{eqn:cost}). Moreover, for sufficiently large $R$, it is an optimal control for the original problem (\ref{eqn:n_sys})-(\ref{eqn:cost}).
\end{proposition}
\begin{proof}
Suppose that $u_e=k(x)$ is the exponentially stabilizing feedback for $x_0\in U$ and let $x_e(t)$ be the corresponding solution. If $R>\max_{t\geqslant0}|x_e(t)|$, the corresponding solution is in the ball $|x|<R$ and the solution satisfies $|x_2|<R$ for $t\geqslant0$ and therefore, $u_e$ exponentially stabilizes $\dot x=\tilde{f}_R(x)+\tilde{g}_R(x)u$, $x(0)=x_0$. Let $J_e(x_0)$ be the value of (\ref{eqn:cost}) for this control. Next, we show that $(\tilde{f}_R, h)$ is zero-state detectable for a set containing $|x|\leqslant \rho$ for sufficiently large $R$. 
To this end, we assume that a solution $x_R(t)$ for $\dot x =\tilde{f}_R(x)$ satisfies $h(x_R(t))=0$ for $t\geqslant0$. 
Then, $|x_R(t)|\leqslant\rho$ from (\ref{ineq:groth_h}). Therefore, for $R>\rho$, $x_R(t)$ is a solution to $\dot x=f(x)$ since $|x_{2R}(t)|\leqslant\rho<R$, where we denote $x_R(t)=(x_{1R}(t),x_{2R}(t))$, $x_{jR}(t)\in\mathbb{R}^{n_j}$, $j=1,2$. 
From the detectability assumption on $(f,h)$, we have $x_R(t)\to0$, $t\to\infty$. Now, the existence of the optimal control $u_R$ follows from Theorem~\ref{thm:minimizer_2}. 

We prove that there exists an $R$ such that the corresponding optimal solution $\bar{x}_R(t)$ satisfies $|\bar{x}_{2R}(t)|<R$ for all $t\geqslant0$. If this is true, the cutoff function has no effect on the OCP defined by $\tilde{f}_R$, $\tilde{g}_R$ and (\ref{eqn:cost}) and therefore $u_R$ is an optimal control for the original problem (\ref{eqn:n_sys})-(\ref{eqn:cost}). Since $\bar{x}_{2R}(t)\to0$, $t\to\infty$ (by \clred{Theorems~\ref{prop:exp_stable}, \ref{thm:minimizer_2}}), define 
\begin{align*}
L_R &:= \max_{t\geqslant0}|\bar{x}_{2R}(t)|, \\
I_R &:= \left\{ t\geqslant0\, |\, |\bar{x}_{2R}(t)|\geqslant L_R/2 \right\}.
\end{align*}
We assume that $L_R\to\infty$ as $R\to\infty$ and derive a contradiction. We may assume that $R\leqslant L_R$. 
Then, for sufficiently large $R$, 
\begin{align*}
J_e(x_0) &\geqslant \int_0^\infty h(\bar{x}_R(t))\,dt\\
    &\geqslant c_h\int_{I_R}|\bar{x}_R(t)|^p\,dt \qquad \text{by (\ref{ineq:groth_h})}\\
    &\geqslant c_h\int_{I_R}|\bar{x}_{2R}(t)|^p\,dt \geqslant c_h |I_R|\left(\frac{L_R}{2}\right)^p,
\end{align*}
\clred{where $c_h$ is independent of $R$,} and thus we have
\begin{equation}
    |I_R|\leqslant C{L_R}^{-p}\leqslant CR^{-p}, \qquad\text{$C$ independent of $R$.}\label{ineq:I_R}
\end{equation}
Next, we give an estimate on the length of $\bar{x}_{2R}(t)$ in $I_R$. To do this, we note from (\ref{ineq:groth_all_app}) that $\tilde{f}_{2R}$, $\tilde{g}_{2R}$ satisfy the estimates 
\clred{
\begin{equation}
|\tilde{f}_{2R}(x_1,x_2)| = O(R^{p+\theta_2}),\quad 
\|\tilde{g}_{2R}(x_1,x_2)\| =O(R^{p/2+\theta_2}) \label{ineq:groth_all_app2}
\end{equation}
}
for $(x_1,x_2)\in\mathbb{R}^n$. Now we compute the length
\begin{align*}
\int_{I_R}\left|\frac{d}{dt}\bar{x}_{2R}(t)\right|\,dt &\leqslant 
    \int_{I_R}|\tilde{f}_{2R}(\bar{x}_R(t))| + \|\tilde{g}_{2R}(\bar{x}_R)\||u_R(t)|\,dt\\
    &\leqslant C|I_R|R^{p+\theta_2} +C'\sqrt{|I_R|}\|u_R\|_{L^2((0,\infty);\mathbb{R}^m)}R^{p/2+\theta_2}\quad\text{by (\ref{ineq:groth_all_app2})}\\
    &\leqslant C'' R^{\theta_2}\leqslant C'' (L_R)^{\theta_2} \quad\text{by (\ref{ineq:I_R})}
\end{align*}
where $C''$ is a positive constant independent of $R$ since $\|u_R\|_{L^2((0,\infty);\mathbb{R}^m)}^2<2J_e(x_0)$. 
If $L_R\to\infty$ as $R\to\infty$, the left side of the above grows at least as $O(L_R)$, which is a contradiction. If $R>\sup_{R>0}L_R$, then $|\bar{x}_{2R}(t)|<R$ for $t\geqslant0$. This choice of $R$ is possible for every $x_0\in U$ and the proposition has been proved. 
\end{proof}
\subsection{Feedback linearizable systems} 
Control system (\ref{eqn:n_sys}) is said to be {\em feedback linearizable} (see \cite{Khalil:92:NS} page 293 or \cite{Isidori:95:NCS} for more detail) in a domain $U\subset\mathbb{R}^n$ containing the origin if there exists a $C^1$ diffeomorphism $T:U\to\mathbb{R}^n$ such that $T(0)=0$ and the change of coordinates $y=T(x)$ transforms (\ref{eqn:n_sys}) into 
\[
\dot y = A y + B\beta(y)^{-1}[u-\alpha(y)]
\]
with $(A,B)$ controllable and $\alpha(y)$, $\beta(y)$ are $C^1$ with $\beta(y)$ being a nonsingular matrix for $y\in T(U)$. If (\ref{eqn:n_sys}) is feedback linearizable in $U$, the input transformation $u=\alpha(y)+\beta(y)v$ brings the system into $\dot y=Ay+Bv$. It is possible to prove that a feedback linearizable system in $U$ is exponentially stabilizable for $U$ in the following way. Take a matrix $K\in\mathbb{R}^{m\times n}$ and an open set $V\subset T(U)$ such that all the solutions of $\dot y=(A+BK)y$ starting from $V$ at $t=0$ stay in $T(U)$ for $t\geqslant0$ and converge to $y=0$ as $t\to\infty$. Also, from the controllability of $(A,B)$, for any $y_0\in T(U)$, there exists a piece-wise continuous control that brings $y_0$ to a point in $V$ within $T(U)$ in finite time. Hence, for any $x_0\in U$ there is a control $u_e$ with which an exponential estimate holds for the corresponding solution of (\ref{eqn:n_sys}). 
\begin{ex}[\cite{Khalil:92:NS} page 296]
A mathematical model of a synchronous generator connected to an infinite bus is 
\begin{align*}
    \frac{d}{dt}\bmat{x_1\\x_2\\x_3}=\bmat{x_2\\ -a[(1+x_3)\sin(x_1+\delta)-\sin\delta]-bx_2\\ 
    -cx_3 +d[\cos(x_1+\delta)-\cos\delta]}+\bmat{0\\0\\1}u
\end{align*}
where $a$, $b$, $c$ and $\delta$ are positive constants. It is feedback linearizable in $U:=\{-\delta<x_1<\pi-\delta\}\subset\mathbb{R}^3$. This system satisfies the growth conditions (\ref{ineq:groth_f}), (\ref{ineq:groth_g}) for $p\geqslant1$, $\theta=0$ in $\mathbb{R}^3$. Therefore, for any $h(x)$ satisfying $h(0)=0$, $Dh(0)=0$ and (\ref{ineq:groth_h}) for some $c_h>0$ and $\rho>0$, an optimal control exists for each initial value in $U$. 
\end{ex}
\begin{ex}[A 2-dimensional pendulum in \cite{Sakamoto:13:automatica,Horibe:17:ieee_cst,Oishi:17:ifac-wc}]
\clred{A pendulum on a cart is a classical control experiment in linear and nonlinear control theories} and if the cart position is neglected, we have a system of the form
\begin{align*}
\dot{x}_1 &=x_2\\
\dot{x}_2 &= \frac{\sin x_1 -x_2^2\sin x_1\cos x_1}{1+\sin^2x_1}-\frac{\cos x_1}{1+\sin^2x_1}u
\end{align*}
where $x_1$ is the angle of the pendulum from vertical (up). For $U:=\{|x_1|<\frac{\pi}{2}\}\subset\mathbb{R}^2$, the system is feedback linearizable and growth conditions (\ref{ineq:groth_f}), (\ref{ineq:groth_g}) are satisfied for $p\geqslant2$, $\theta=0$ and thus, for a suitable $h$, an optimal control exists in $U$. In \cite{Horibe:17:ieee_cst,Oishi:17:ifac-wc}, however, exponentially stabilizing feedback swing-up controls (namely $(x_1(0),x_2(0))=(\pi,0)$) are designed by computing stable manifolds of associated Hamiltonian systems. Moreover, it has been reported that for $h(x)=\varepsilon (|x_1|^2+|x_2|^2)$ with $\varepsilon$ small, a number of feedback swing-up stabilization controls exist for a single cost functional (\ref{eqn:cost}). Each control effectively uses reaction (or swing) of the pendulum and the more swings are used during the control, the smaller the cost value gets. 
In \cite{Horibe:17:ieee_cst}, a detailed account of this phenomenon (namely, the existence of multiple solutions to an HJBE) is shown using 3D figures of the stable manifold. However, it has not been confirmed whether or not an optimal control that minimizes the cost functional exists. Now, using Theorem~\ref{thm:minimizer_2}, we can answer this question by saying that as long as $\varepsilon>0$, there exists an optimal control that achieves the minimum value of the cost. The question of whether or not infinitely many swing-up controllers for the HJBE exist, when $h=0$, is still open. 
\end{ex}
\subsection{Systems with globally exponentially stable zero dynamics}
It is known (see, e.g., \cite{Byrnes:91:ieee-ac}, \cite{Isidori:95:NCS}) that under suitable conditions such as relative degree condition, system (\ref{eqn:n_sys}), when it is a single input system, can be transformed by a smooth coordinate change and a smooth feedback transformation into 
\begin{subequations}
\begin{align}
&\dot\eta = q(\eta,\xi_1)\label{eqn:zero_dyn1}\\
& \dot{\xi}_1 = \xi_2, \ldots, \dot{\xi}_r = v \label{eqn:zero_dyn2}
\end{align}
\end{subequations}
where $\xi_j(t)\in\mathbb{R}^1$, $j=1,\ldots,r$, $\eta(t)\in\mathbb{R}^{n-r}$ and $v\in\mathbb{R}^1$ is a new input. This system transformation is important to understand the system structure from an input-output viewpoint and $\dot\eta=q(0,\eta)$ is called zero dynamics describing the dynamics on which the system output $y=\xi_1$ is identically zero. 

In this subsection, we are interested merely in an OCP for (\ref{eqn:zero_dyn1})-(\ref{eqn:zero_dyn2}) with input $v$ and a cost functional. 
Suppose that zero dynamics $\dot\eta=q(0,\eta)$ is globally exponentially stable. Then, there exists a smooth feedback control that exponentially stabilizes (\ref{eqn:zero_dyn1})-(\ref{eqn:zero_dyn2}) for all initial conditions in $\mathbb{R}^n$. Suppose, in addition, that for any $R>0$ there is a $C=C(R)>0$ such that $q(\eta, \varphi_R(\xi_1)\xi_1)$ satisfies a growth condition
\[
|q(\eta,\varphi_R(\xi_1)\xi_1)|\leqslant C|\eta|^p
\]
for sufficiently large $\xi_1$ and $\eta$, where $\varphi_R$ is the cutoff function and $p>0$ is a constant independent of $R$. 
We consider an OPC (\ref{eqn:zero_dyn1})-(\ref{eqn:zero_dyn2}) and $J=\int_0^\infty |v|^2/2 +h(\xi_1,\ldots,\xi_r,\eta)\,dt$, where $h\geqslant0$ is $C^1$ with $h(0)=0$, $Dh(0)=0$ and satisfies (\ref{ineq:groth_h}) for $p>0$ above and some constants $c_h>0$, $\rho>0$. Assume, in addition, that $h$ is zero-state detectable with (\ref{eqn:zero_dyn1})-(\ref{eqn:zero_dyn2}) in $\mathbb{R}^n$. 
Then, from Proposition~\ref{prop:minimizer_2_appl}, \clred{an optimal control} exists for the OCP for any initial conditions. 
\begin{ex}
Consider a 3-dimensional system
\begin{align*}
\dot x_1 &=-x_1 +x_1^2x_2\\
\dot x_2 &=x_3 \\
\dot x_3 &=u.
\end{align*}
This system is in the form (\ref{eqn:zero_dyn1})-(\ref{eqn:zero_dyn2}) with globally exponentially stable zero dynamics $\dot x_1=-x_1$ and is globally exponentially stabilizable. Consider also an OCP for this system 
with a quadratic cost functional $J=\frac{1}{2}\int_0^\infty u^2+x_1^2+x_2^2+x_3^2\,dt$. This problem does not satisfy the growth condition in Theorem~\ref{thm:minimizer_2}, but, with a cutoff function on $x_2$, the right side of the first equation $-x_1+\varphi_R(x_2)x_1^2x_2$ satisfies the hypotheses in Proposition~\ref{prop:minimizer_2_appl}. Therefore, a solution for this OCP exists \clred{for any $x_0\in\mathbb{R}^3$}. 
\end{ex}
%
%
\subsection{Turnpike analysis}
Let us consider a finite horizon OCP for (\ref{eqn:n_sys}) with 
\begin{equation}
J_T = \int_0^T |u(t)|^2/2 +h(x(t))\,dt,\label{cost:J_T}
\end{equation}
where $h(x)=x^\top C^\top Cx$, $C\in \mathbb{R}^{r\times n}$, and $x(T)=x_f$ is specified. 

Suppose that $u_T$ is an optimal control and $x_T$ is the optimal trajectory. 
The optimal control $u_T$ is said to {\em enjoy the turnpike property} if for any $\varepsilon>0$, there exists an $\eta_\varepsilon>0$ such that 
\[
\left| \left\{t\geqslant0\,|\,|u_T (t)|+|x_T(t,x_0)|>\varepsilon  \right\} \right| <\eta_\varepsilon
\]
for all $T>0$, where $\eta_\varepsilon$ depends only on $\varepsilon$, $f$, $g$, $h$, and $x_0$ and $|\cdot|$ denotes length (Lebesgue measure) of interval. A necessary condition for optimality is the existence of $(x(t),p(t))$, $0\leqslant t\leqslant T$ for the Hamiltonian system (\ref{eqn:hsys1})-(\ref{eqn:hsys2}) with $x(0)=x_0$, $x(T)=x_f$.

In \cite{sakamoto:20:prep_cdc}, it is shown that one of the sufficient conditions for the OCP to have a solution with turnpike property is that at an equilibrium of the Hamiltonian system, a stable manifold $\Lambda_S$ exists and satisfies $x_0\in \mathrm{Int}(\pi(\Lambda_S))$ and an unstable manifold $\Lambda_U$ exists and satisfies $x_f\in\mathrm{Int}(\pi(\Lambda_U))$, where $\mathrm{Int}$ represents the interior of a set in $\mathbb{R}^n$ \clred{and $\pi$ is the canonical projection appeared in \S~\ref{sctn:stabl_manifold}}. We show, in the following example, that those conditions can be examined using the results of the present paper. 
\begin{ex}
Backstepping design is one of the popular and powerful feedback design methods which is applied widely in practice
(see, e.g., \cite{Byrnes:89:syscon,Krstic:95:NACD,Sepulchre:97:CNC}). In this section, we show a class of nonlinear control systems for which \clred{the} turnpike occurs for all initial and terminal conditions using backstepping stabilization. 

Let us consider a nonlinear control system of the form
\begin{subequations}\label{eqns:backstepp}
\begin{align}
&\dot x_1 = f_1(x_1)+g_1(x_1)x_2 \label{eqn:backstepp1}\\
&\dot x_2 =f_2(x_1,x_2)+g_2(x_1,x_2)u,\label{eqn:backstepp2}
\end{align}
\end{subequations}
where $x_1,\ x_2\in\mathbb{R}^{n}$ are the \clred{states} and $u\in\mathbb{R}^n$ is the control. We assume that $(0,0)$ is an equilibrium of (\ref{eqn:backstepp1})-(\ref{eqn:backstepp2}), that $f_j$, $g_j$, $j=1,2$ are smooth and that $g_1(x_1)$, $g_2(x_1,x_2)$ are invertible for all $x_1$, $x_2$.

The backstepping technique is a cascade design, in which (\ref{eqn:backstepp1}) is stabilized with a virtual stabilizing input $x_2=\alpha(x_1)$ and then obtain a feedback control that stabilizes the error system for $z:=x_2-\alpha(x_1)$. 
A typical feedback stabilization using backstepping under the above conditions is, first, obtain $\alpha(x_1)$ that globally exponentially stabilizes (\ref{eqn:backstepp1}) with a Lyapunov function $V_1(x_1)$, such as $\alpha(x_1)=g_1(x_1)^{-1}[-f_1(x_1)-x_1]$. 
Next, with a new coordinates $(x_1,z)$ where $z=x_2-\alpha(x_1)$, the overall exponentially stabilizing feedback is given as 
\[
u=g_2(x_1,z+\alpha(x_1))^{-1}\left[-f_2(x_1,z+\alpha(x_1)) +\dot{\alpha}(x_1) -g_1(x_1)^\top DV_1(x_1)^\top -z\right].
\]
For above $\alpha(x_1)$, one takes $V_1(x_1)=|x_1|^2/2$ and the global exponential stability of the closed loop system is guaranteed with 
$V(x_1,z)=V_1(x_1)+|z|^2/2$ and $\dot{V}=-|x_1|^2-|z|^2$. 
The time-reversed system of (\ref{eqn:backstepp1})-(\ref{eqn:backstepp2}) is 
\begin{subequations}\label{eqns:backstepp_rev}
\begin{align}
& x_1' = -f_1(x_1)-g_1(x_1)x_2 \label{eqn:backstepp_rev1}\\
&x_2' = -f_2(x_1,x_2)-g_2(x_1,x_2)u,\label{eqn:backstepp_rev2}
\end{align}
\end{subequations}
where $(\cdot)' = \frac{d}{d\tau}$, $\tau =-t$. This system can also be globally exponentially stabilized via the backstepping method. Thus for OCPs with cost functionals satisfying growth conditions in Theorem~\ref{thm:minimizer_2} or Proposition~\ref{prop:minimizer_2_appl} \clred{and the linear detectabiliity condition}, one can give estimates on the existence regions of stable and unstable manifolds \clred{$\Lambda_S$, $\Lambda_U$} in associated Hamiltonian systems via Theorem~\ref{thm:smani} \clred{as
$
\pi(\Lambda_S)=\mathbb{R}^n$, $\pi(\Lambda_U)=\mathbb{R}^n$.
}%
\clred{The result in \cite{sakamoto:20:prep_cdc} guarantees that the turnpike occurs for all initial and terminal states. }

For example, a nonlinear system
\begin{align*}
    \dot{x}_1 &=x_1^2 +(1+x_1^2)x_2\\
    \dot{x}_2 &= x_2^2 +u
\end{align*}
satisfies the conditions above and is globally exponentially stabilizable at the origin via a smooth backstepping feedback. For a cost functional 
\[
J = \frac{1}{2}\int_0^\infty u^2+x_1^2+x_2^2\,dt, 
\]
one readily sees that Proposition~\ref{prop:minimizer_2_appl} can be applied \clred{(note that the growth conditions in Theorem~\ref{thm:minimizer_2} are not satisfied)}.
Then, using Theorem~\ref{thm:smani} for this OCP, the stable manifold $\Lambda_S$ of the associated Hamiltonian system at the origin exists with $\pi(\Lambda_S)=\mathbb{R}^{2}$ and applying the same argument for the time-reversed OCP replacing $t$ by $\tau$, the unstable manifold $\Lambda_U$ of the Hamiltonian system at the origin exists with $\pi(\Lambda_U)=\mathbb{R}^{2}$. Therefore, we conclude that for all initial and terminal states the finite horizon OCP
\[
J = \frac{1}{2}\int_0^T u^2+x_1^2+x_2^2\,dt, 
\]
with sufficiently large $T$ has an optimal control that enjoys the turnpike property. 
\end{ex}
\begin{remark}
In \cite{sakamoto:20:prep_cdc}, a numerical example is worked out in detail in which turnpike occurs for all $x_0$ and for sufficiently small $x_f$ using the global exponential stabilizability. The example also exhibits the peaking phenomenon \cite{Sussmann:91:ieee_tac} during the turnpike. 
\end{remark}

{\bf Acknowledgment.} The author would like to thank the Chair of Computational Mathematics at University of Deusto for their support during his stay. 
He is also grateful to Dario Pighin for valuable discussions. Especially, the proofs of Proposition~\ref{prop:minimizer} and Lemmas~\ref{lemma:ode_affine}, \ref{lemma:h_convergence} owe to him.

\appendix
\section{Useful results}
\setcounter{equation}{0}
\renewcommand{\theequation}{\Alph{section}.\arabic{equation}}
\renewcommand{\thetheorem}{\Alph{section}.\arabic{theorem}}
This Appendix includes several preliminary propositions and lemmas for the proofs of the main results. 
\begin{proposition}\label{prop:ode_new}
Let $D\subset\mathbb{R}^{n+1}$ be an open set and let $F:D\to\mathbb{R}^n$, $(t,x)\in D \mapsto F(t,x)$, be a map which is measurable in $t$ and continuous in $x$. For positive $\alpha$, $\beta$, let us introduce notations $I_\alpha(t_0)$, $B_\beta(x_0)$ representing the sets 
\[
I_\alpha(t_0)=\{t\geqslant0\,|\, |t-t_0|\leqslant \alpha\}, \quad
B_\beta(x_0)=\{x\in\mathbb{R}^n\,|\,|x-x_0|\leqslant\beta\}.
\]
Let $U\subset\mathbb{R}^n$ be a compact set and assume that one can take $\bar\beta>0$ such that 
\[
V:=[0,\infty)\times \bigcup_{x_0\in U}B_{\bar{\beta}}(x_0)\subset D.
\]
Assume also that there exist $0<\alpha$, $0<\beta\leqslant\bar\beta$ and functions $m(t)$, $k(t)$ such that the following are satisfied, 
\begin{subequations}
\begin{gather}
|F(t,x)|\leqslant m(t) \quad \text{on }V,\quad \text{and}\quad
\int_{I_\alpha(t_0)}m(t)\,dt \leqslant \beta \quad\text{for all } t_0\geqslant0, \label{ineq:ode_new1}\\
|F(t,x)-F(t,y)|\leqslant k(t)|x-y| \quad \text{on }V,\quad \text{and}\quad
\int_{I_\alpha(t_0)}k(t)\,dt \leqslant\frac{1}{2} \quad\text{for all } t_0\geqslant0.\label{ineq:ode_new4}
\end{gather}
\end{subequations}
Then, for any $t_0\geqslant0$, $x_0\in U$, there exists a unique solution $x(t,t_0,x_0)$ to $\dot x=F(t,x)$ passing through $(t_0,x_0)$ defined on $I_\alpha(t_0)$ satisfying $|x(t,t_0,x_0)-x_0|\leqslant\beta$ for $t\in I_\alpha(t_0)$, where $\alpha$, $\beta$ are independent of $t_0$. This $x$ is a continuous map defined on $\bigcup_{t_0\geqslant0} (I_\alpha(t_0)\times \{t_0\})\times U$. 
\end{proposition}
\begin{proof}
Define a closed set $\mathscr{A}$ in $C^0(I_\alpha(0);\mathbb{R}^n)$ and an operator $T:\mathscr{A}\to C^0(I_\alpha(0);\mathbb{R}^n)$ by
\begin{gather*}
\mathscr{A}=\{\varphi\in C^0(I_\alpha(0);\mathbb{R}^n)\,|\, \varphi(0)=0, \ |\varphi(t)|\leqslant\beta \text{ for } t\in I_\alpha(0)\}\\
(T\varphi)(t)=\int_{t_0}^{t_0+t}F(s,\varphi(s-t_0)+x_0)\,ds.
\end{gather*}
Then for $t\in I_\alpha(0)$, we have
\begin{align*}
|(T\varphi)(t)| &\leqslant \int_{t_0}^{t_0+t}|F(s,\varphi(s-t_0)+x_0)\,ds\\
&\leqslant \int_{t_0}^{t_0+t}m(s)\,ds\\
&\leqslant \int_{I_\alpha(t_0)}m(s)\,ds\leqslant\beta,
\end{align*}
where we have used $(s,\varphi(s-t_0)+x_0)\in V$ for $s\in [t_0,t_0+t]$ and (\ref{ineq:ode_new1}), implying that $T\varphi \in \mathscr{A}$. Since $(T\varphi)(0)=0$, we confirm that $T:\mathscr{A}\to \mathscr{A}$. 

Next, take $\varphi, \bar \varphi \in \mathscr{A}$ and $t\in I_\alpha(0)$. Then, 
\begin{align*}
|(T\varphi)(t)-(T\bar\varphi)(t)| &\leqslant \int_{t_0}^{t_0+t}|F(s,\varphi(s-t_0)+x_0)-F(s,\bar\varphi(s-t_0)+x_0)|\,ds\\
&\leqslant \int_{t_0}^{t_0+t}k(s)|\varphi(s-t_0)-\bar\varphi(s-t_0)|\,ds\\
&\leqslant\|\varphi-\bar\varphi\|_{C^0(I_\alpha(0);\mathbb{R}^n)}\int_{I_\alpha(t_0)}k(s)\,ds\\
&\leqslant \frac{1}{2}\|\varphi-\bar\varphi\|_{C^0(I_\alpha(0);\mathbb{R}^n)},
\end{align*}
where we have used $(s,\varphi(s-t_0)+x_0),\ (s,\bar\varphi(s-t_0)+x_0)\in V$ for $s\in [t_0,t_0+t]$ and (\ref{ineq:ode_new4}). Therefore, by Contraction Mapping Theorem, there exists a unique $\varphi\in\mathscr{A}$ such that 
\[
\varphi(t)=\int_{t_0}^{t_0+t}F(s,\varphi(s-t_0)+x_0)\,ds.
\]
Defining $x(t)=\varphi(t-t_0)+x_0$, $x(t)$ is the unique solution of $\dot x=F(t,x)$ passing through $(t_0,x_0)$ defined on $I_\alpha(t_0)$, where $\alpha$ is independent of $t_0$. 

If we regard $T=T_{(t_0,x_0)}$, we have proved that $T_{(t_0,x_0)}$ is a uniform contraction with respect to $(t_0,x_0)$ in $[0,\infty)\times U$. Therefore, from Uniform Contraction Mapping Theorem, the fixed point $\varphi$ is continuous function of $(t_0,x_0)$ and therefore, $x(t,t_0,x_0)$ is a continuous function in $t\in I_\alpha(t_0)$ and $(t_0,x_0)\in [0,\infty)\times U$, where $\alpha$ is independent of $t_0$. 
\end{proof}
\begin{lemma}\label{lemma:ode_affine}
\clred{S}uppose that $f:\mathbb{R}^n\to\mathbb{R}^n$ is $C^1$ and that $g:\clred{\mathbb{R}^n\to\ }\mathbb{R}^{n\times m}$ is Lipschitz continuous. Let $U\subset\mathbb{R}^n$ be a compact set and let 
\[
V:=[0,\infty)\times \bigcup_{x_0\in U}B_\beta(x_0). 
\]
\clred{Let us consider (\ref{eqn:n_sys}) with $u\in L^2((0,\infty);\mathbb{R}^m)$.} Then, there exist $\alpha$, $\beta>0$ and $m(t)$, $k(t)$ such that (\ref{ineq:ode_new1})-(\ref{ineq:ode_new4}) in Proposition~\ref{prop:ode_new} \clred{with $F(t,x)=f(x)+g(x)u(t)$} are satisfied. Moreover, there exists an $M>0$ which is independent of $t_0$ such that \clred{the solution satisfies}
\[
|x(t,t_0,x_0)|\leqslant M \quad \text{for }t\in I_\alpha(t_0), \ t_0\geqslant0. 
\]
\end{lemma}
\begin{proof}
Let $B_{\beta,U}:=\bigcup_{x_0\in U}B_\beta(x_0)$. 
We show that the conditions in Proposition~\ref{prop:ode_new} are satisfied. First, 
\begin{align*}
|f(x)+g(x)u(t)|&\leqslant|f(x)|+\|g(x)\|\,|u(t)|\\
&\leqslant \sup_{x\in B_{\beta,U}}|f(x)|+\sup_{x\in B_{\beta,U}}\|g(x)\|\,|u(t)|:=m(t).
\end{align*}
Then, for $t_0\geqslant0$, 
\begin{align}
\int_{I_\alpha(t_0)}m(t)\,dt &\leqslant 2\alpha\sup_{x\in B_{\beta,U}}|f(x)|+\sup_{x\in B_{\beta,U}}\|g(x)\|\int_{I_{\alpha}(t_0)}|u(t)|\,dt\nonumber\\
&\leqslant 2\alpha \sup_{x\in B_{\beta,U}}|f(x)|+\sqrt{\alpha}\sup_{x\in B_{\beta,U}}\|g(x)\|\|u\|_{L^2((0,\infty);\mathbb{R}^m)}.\nonumber
\end{align}
Also, we have
\begin{align}
|f(x)+g(x)u(t)-f(y)-g(y)u(t)|&\leqslant 
    |f(x)-f(y)|+\|g(x)-g(y)\|\,|u(t)|\nonumber\\
&\leqslant \sup_{x\in B_{\beta,U}}\|Df(x)\|\,|x-y| +L_g|x-y|\,|u(t)|:=k(t),\nonumber
\end{align}
where $L_g$ is a Lipschitz constant for $g$ in $B_{\beta,U}$. Then, for $I_\alpha(t_0)$, $t_0\geqslant0$,
\[
\int_{I_\alpha(t_0)}k(t)\,dt \leqslant 2\alpha \sup_{x\in B_{\beta,U}}\|Df(x)\|+\sqrt{\alpha}L_g\|u\|_{L^2((0,\infty);\mathbb{R}^m)}.
\]
We can find $\alpha$, $\beta>0$ independently of $t_0$ satisfying 
\begin{gather*}
2\alpha \sup_{x\in B_{\beta,U}}|f(x)|+\sqrt{\alpha}\sup_{x\in B_{\beta,U}}\|g(x)\|\|u\|_{L^2((0,\infty);\mathbb{R}^m)}<\beta,\\
2\alpha \sup_{x\in B_{\beta,U}}\|Df(x)\|+\sqrt{\alpha}L_g\|u\|_{L^2((0,\infty);\mathbb{R}^m)}<\frac{1}{2},
\end{gather*}
showing that all the conditions in Proposition~\ref{prop:ode_new} are satisfied. Since $|x(t,t_0,x_0)-x_0|<\beta$ for $t\in I_{\alpha}(t_0)$ and we can take an $r>0$ such that if $x_0\in U$ then $|x_0|<r$, 
\[
|x(t,t_0,x_0)|\leqslant |x_0|+\beta<r+\beta,
\]
the right side of which is independent of $t_0$. 
\end{proof}
%
\begin{lemma}\label{lemma:h_convergence}
Let $H:\mathbb{R}^n\to\mathbb{R}$ be a nonnegative locally Lipschitz function, which is coercive. 
Let $x(t)$ be a solution to (\ref{eqn:n_sys}) such that $\int_0^\infty H(x(t))\,dt<\infty$ for a $u\in L^2((0,\infty);\mathbb{R}^m)$. Then, $x(t)$ is bounded for $t\geqslant0$ and $H(x(t))\to0$ as $t\to\infty$.
\end{lemma}
\begin{proof}
From the assumptions on $H$, there exist constants $R>0$, $\mu>0$ such that
\begin{align}
|x|> R \ \Rightarrow \mu< H(x).\label{ineq:outsideR}
\end{align}
Take the compact set $U$ in Lemma~\ref{lemma:ode_affine} as $U=B_{R}(0)$. Take also $\alpha$ and $M$ in the same proposition. We note that if $x(\bar t)\in U$ for some $\bar t >0$, then, 
\begin{equation}
|x(t)|<M \quad \text{for }t\in I_\alpha(\bar t),\label{ineq:bunded_M}
\end{equation}
where $M$ and $\alpha$ are independent of $\bar t$. 
We next prove the following claim. If a nonnegative integrable function $\psi(t)$ satisfies $\int_0^\infty\psi(t)\,dt<\infty$, then, for any $l>0$ and $\varepsilon>0$, there exists a $T_{l,\varepsilon}>0$ such that for any $\tau >T_{l,\varepsilon}$, there exists a $\bar t\in [\tau-\varepsilon,\tau+\varepsilon]$ such that $\psi(\bar t)<l$. Suppose that the claim is not true. Then, there exist $l>0$ and $\varepsilon>0$ such that for any $T>0$, there exists a $t>T$ such that $\psi(t)\geqslant l$ for all $t\in [t-\varepsilon,t+\varepsilon]$. From this, we take a sequence $\{t_n\}$, $t_n>0$, $t_n +2\varepsilon<t_{n+1}$, $t_n\to\infty$ $(n\to\infty)$ such that 
$\psi(t)\geqslant l$ for all $t\in [t_n-\varepsilon,t_n+\varepsilon]$. Then we have 
\[
\int_0^\infty \psi(t)\,dt> \sum_{n=1}^\infty \int_{t_n-\varepsilon}^{t_n+\varepsilon}\psi(t)\,dt>\sum_{n=1}^\infty 2l\varepsilon=\infty,
\]
which is a contradiction. Now we apply this claim with $l=\mu$ and $\varepsilon=\alpha$. Then, there exists a $T_{\mu,\alpha}>0$ such that for any $\tau>T_{\mu,\alpha}$, there exists a $\bar t\in [\tau-\alpha,\tau+\alpha]$ such that $H(x(\bar t))\leqslant\mu$. Then, from (\ref{ineq:outsideR}), we have $|x(\bar t)|\leqslant R$, equivalently, $x(\bar t)\in U$. Then, from (\ref{ineq:bunded_M}), 
\begin{equation}
|x(t)|<M \quad \text{for }t\in I_\alpha(\bar t).\label{ineq:boundM2}
\end{equation}
But, $\bar t\in [\tau-\alpha,\tau+\alpha]$ implies that $\tau\in I_\alpha(\bar t)$. Then, from (\ref{ineq:boundM2}), $|x(\tau)|<M$. We emphasize that $M$ and $\alpha$ are independent of $\bar t$ and $\tau$ is any sufficiently large time (larger than $T_{\mu,\alpha}$). Therefore, on $[T_{\mu,\alpha},\infty)$, $x(t)$ is bounded. Since $x(t)$ is continuous, it is bounded for $[0,\infty)$ and thus $H(x(t))$ is uniformly continuous on $[0,\infty)$ since $H$ is locally Lipschitz. From Barbalat's Lemma (see, e.g., \cite{Popov:73:HCS}), we have $H(x(t))\to0$ as $t\to\infty$. 
\end{proof}
\begin{remark}
Lemma~\ref{lemma:h_convergence} can be seen as a nonlinear finite-dimensional modification of Datko-Pazy Theorem (see \cite{Pazy:83:SLOAPDE}, Page~116, Theorem~4.1). Namely, in addition to the hypothes in Lemma~\ref{lemma:h_convergence}, assume that for any $\varepsilon>0$, there exists a $\delta_\varepsilon>0$ such that $H(x)<\delta_\varepsilon$ if $|x|<\varepsilon$. Then, $x(t)\to0$ as $t\to\infty$.
\end{remark}
The following Proposition is a generalization of Lemma~3.1 on Page~24 in \cite{Hale:73:ODE} for Carath\'eodory solution. 
\begin{proposition}\label{prop:hale_lemma}
Let $D\subset \mathbb{R}^{n+1}$ be an open set. For $m=0,1,\ldots,$ let $f_m:D\to\mathbb{R}^n$ be maps such that $f_m(t,x)$ are measurable in $t$ for each $x$ and continuous in $x$ for each $t$. 
For $f_m(t,x)$, $m=0,1,2,\ldots,$ assume the following hypotheses. 
For any compact $\bar D\subset D$, there exist integrable $k_0(t)$, $k_m(t)$, $\bar{k}_m(t)$, $m=1,2,\ldots$, such that
\begin{align*}
&|f_0(t,x)|\leqslant k_0(t)\quad \text{for}\quad (t,x)\in \bar D,\\
&|f_m(t,x)-f_0(t,m)|\leqslant k_m(t)\quad \text{for}\quad (t,x)\in \bar D,\\
&\int_Ik_m(t)\,dt \to 0 \ (m\to\infty)\quad\text{for any interval }I\subset\bar D,\\
& \int_I k_m(t)\,dt\leqslant C |I| \quad\text{for any interval }I\subset\bar D, \\
&\qquad \qquad\qquad\qquad    \text{ where $C>0$ is a constant independent of $m$},\\
& |f_m(t,x)-f_m(t,y)|\leqslant\bar{k}_m(t)|x-y|\quad \text{for}\quad (t,x), (t,y)\in \bar D,\\
& \int_I\bar{k}_m(t)\,dt \quad \text{is bounded for }m\in\mathbb{N}, I\subset\bar D.
\end{align*}
Let $x_m\in\mathbb{R}^{n}$, $m=0,1,2,\ldots$,  be points such that $x_m\to x_0$ when $m\to\infty$. Let finally $\varphi_m$ be a solution for 
\[
\dot x=f_m(t,x)
\]
passing through $(t_0,x_m)$, $m=0, 1,2,\ldots$. 
If $\varphi_0(t)$ is a unique solution defined on a finite interval $[a,b]$, 
then, for sufficiently large $m$, $\varphi_m$ is defined on $[a,b]$ and uniformly converges to $\varphi_0$ on $[a,b]$ as $m\to\infty$. 
\end{proposition}
\begin{proof}
Let $S:=\{(t,\varphi_0(t))\,|\,t\in[a,b]\}$ and $U$ be a compact set in $D$ which contains $S$ in its interior; $S\subset U\subset D$. By the hypotheses, there exist integrable $k_m(t)$ in $U$, $m=0,1,2,\ldots$, such that 
\begin{align}
&|f_0(t,m)|\leqslant k_0(t)\quad \text{for }(t,x)\in U, \label{ineq:f_0}\\
&|f_m(t,x)-f_0(t,x)|\leqslant k_m(t)\quad \text{for }(t,x)\in U, \ m=1,2,\ldots,\label{ineq:fm-f0}\\
&\int_I k_m(t)\,dt\to0 \quad \text{as }m\to\infty \text{ for } I\subset U, \label{eqn:int_km_converg}\\
&\int_I k_m(t)\,dt\leqslant C|I| \text{ for } I\subset U, \label{ineq:km_int_uniform_bdd}
\end{align}
where $C>0$ is a constant independent of $m$. 
Take $\alpha>0$, $\beta>0$ such that 
\begin{gather}
I_\alpha(\bar t)\times B_{\beta}(\bar x) \subset U \text{ for }(\bar{t},\bar{x})\in S,\nonumber\\
\int_I k_0(s)+k_m(s)\,ds\leqslant\beta  \text{ for } I\subset I_\alpha(t_0),\ m\in\mathbb{N}.\label{eqn:int_k0_km_beta}
\end{gather}
For sufficiently large $m$, $I_\alpha(t_0)\times B_\beta(x_m)\subset U$ holds and $\varphi_m(t)$ is defined on $I_\alpha(t_0)$ satisfying $|\varphi_m(t)-x_m|\leqslant\beta$ for $t\in I_\alpha(t_0)$, which is from 
\begin{align*}
|f_m(t,x)| &\leqslant |f_0(t,x)|+k_m(t)\\
    &\leqslant k_0(t)+k_m(t)\quad\text{by (\ref{ineq:f_0})}
\end{align*}
and (\ref{eqn:int_k0_km_beta}). This shows $(t,\varphi_m(t))\in I_\alpha(t_0)\times B_\beta(x_m)\subset U$ and $\varphi_m(t)$ is uniformly bounded. 
Moreover, for $t_1<t_2$ in $I_\alpha(t_0)$, 
\begin{align*}
|\varphi_m(t_2)-\varphi_m(t_1)| &\leqslant \int_{t_1}^{t_2}|f_m(s,\varphi_m(s))|\,ds\\
&\leqslant\int_{t_1}^{t_2}k_0(s)+k_m(s)\,ds\leqslant \int_{t_1}^{t_2}k_0(s)\,ds+C(t_2-t_1)
\end{align*}
by (\ref{ineq:km_int_uniform_bdd}) 
and therefore $\varphi_m$ is equicontinuous. By using Ascoli-Arzel\'a Theorem, up to subsequence, we have 
\[
\varphi_m\to\bar{\varphi}\quad \text{uniformly on }I_\alpha(t_0)
\]
for some $\bar \varphi\in C(I_\alpha(t_0);\mathbb{R}^n)$. We will show $\bar \varphi=\varphi_0$ by using the integral equation for $\varphi_m$;
\begin{equation}
\varphi_m(t)=x_m+\int_{t_0}^tf_m(s,\varphi_m(s))\,ds.\label{eqn:int_eq_phim}
\end{equation}
Using the hyphotheses, let $\bar{k}_m(t)$ be integrable functions on $I_\alpha(t_0)$, $m=1,2,\ldots$, such that
\begin{gather}
|f_m(t,x)-f_m(t,y)|\leqslant \bar{k}_m(t)|x-y|
\quad \text{for }(t,x), (t,y)\in U\label{ineq:fm_kbar}\\
\int_{t_0}^t\bar{k}_m(s)\,ds \text{ is bounded for }m
 \text{ and }t\in I_\alpha(t_0).\label{eqn:int_kbar_bounded}
\end{gather}
Then, it follows that, for $t_0\leqslant t\leqslant t_0+\alpha$, 
\begin{align*}
&\left| \int_{t_0}^tf_m(s,\varphi_m(s))\,ds 
    - \int_{t_0}^tf_0(s,\bar{\varphi}(s))\,ds\right|\\
&\quad  \leqslant\int_{t_0}^t|f_m(s,\varphi_m(s))-f_0(s,\bar{\varphi}(s))|\,ds\\
&\quad  \leqslant\int_{t_0}^t 
    |f_m(s,\varphi_m(s))-f_m(s,\bar{\varphi}(s))|\,ds 
    +\int_{t_0}^t |f_m(s,\bar{\varphi}(s))-f_0(s,\bar{\varphi}(s))|\,ds \\ 
&\quad  \leqslant\int_{t_0}^t \bar{k}_m(s)|\varphi_m(s)-\bar{\varphi}(s)|\,ds
+\int_{t_0}^tk_m(s)\,ds\quad \text{by (\ref{ineq:fm-f0}), (\ref{ineq:fm_kbar})}\\
&\quad  \leqslant \|\varphi_m-\bar{\varphi}\|_{C(I_\alpha(t_0);\mathbb{R}^n)}\int_{t_0}^{t}\bar{k}_m(s)\,ds +\int_{t_0}^tk_m(s)\,ds\\
&\quad \to0 \ (m\to\infty),
\end{align*}
where we have used the uniform convergence of $\varphi_m$ to $\bar \varphi$, (\ref{eqn:int_km_converg}) and (\ref{eqn:int_kbar_bounded}). Now, taking the limit $m\to\infty$ in (\ref{eqn:int_eq_phim}) shows that $\bar\varphi$ is a solution for $\dot x=f_0(t,x)$ passing through $(t_0,x_0)$. Thus, the uniqueness of initial value problem implies that $\bar{\varphi}=\varphi_0$. This means that for every convergent subsequence of $\{\varphi_m\}$ on $I_\alpha(t_0)$ converges to $\varphi_0$ and therefore, $\varphi_m$ uniformly converges to $\varphi_0$ on $I_\alpha(t_0)$. 

Replacing $x_m$ with $\varphi_m(t_0+\alpha)$, $x_0$ with $\varphi_0(t_0+\alpha)$ and $t_0$ with $t_0+\alpha$ and applying the above procedure, one obtains $\varphi_m(t)$ defined on $[t_0-2\alpha, t_0+2\alpha]$ that is uniformly convergent to $\varphi_0$. Repeating this, the uniform convergence of $\varphi_m(t)$ to $\varphi_0(t)$ on $[a,b]$ is proved. 
\end{proof}
%


\end{document}